\documentclass[reqno,11pt]{amsart}

\usepackage{amsmath,cancel}
\usepackage{amsfonts,enumitem}
\usepackage{amssymb}
\usepackage{color}
\usepackage{mathrsfs}
\usepackage[dvipsnames]{xcolor}
\usepackage[colorlinks=true]{hyperref}

\usepackage{cleveref}

\newtheorem{thm}{Theorem}[section]
\newtheorem{lem}[thm]{Lemma}
\newtheorem{cor}[thm]{Corollary}
\newtheorem{prop}[thm]{Proposition}

\newtheorem{rem}[thm]{Remark}

\numberwithin{equation}{section}

\newcommand{\R}{\mathbb{R}}
\newcommand{\C}{\mathbb{C}}

\newcommand{\ml}{\mathcal{L}}

\newcommand{\ve}{\varepsilon}
\newcommand{\rd}{\mathrm{d}}
\newcommand{\dom}{\mathrm{dom}}

\newcommand{\Gk}{\mathcal{G}_\kappa}
\newcommand{\Tk}{\mathcal{T}}

\newcommand{\Y}{\mathsf{X}}

\newcommand{\nsa}{n^*}
\newcommand{\nsap}{n_+^*}
\newcommand{\nsam}{n_-^*}
\newcommand{\nsapm}{n_\pm^*}
\newcommand{\nsb}{n_*}
\newcommand{\nsbp}{n_*^+}
\newcommand{\nsbm}{n_*^-}
\newcommand{\nsbpm}{n_*^\pm}

\newcommand{\dhr}{\mathrel{\lhook\joinrel\relbar\kern-.8ex\joinrel\lhook\joinrel\rightarrow}}

\begin{document}

\title[Singular or Degenerate Operators in Weighted $L_1$-spaces]{Analytic Semigroups  in Weighted $L_1$-Spaces on the Half-Line Generated by Singular or Degenerate Operators}

\author{Patrick Guidotti}
\address{Department of Mathematics\\ University of California at Irvine\\ 340 Rowland Hall\\ Irvine, CA 92697-3875, USA}
\email{pguidott@uci.edu}
\author{Philippe Lauren\c{c}ot}
\address{Laboratoire de Math\'ematiques (LAMA) UMR~5127, Universit\'e Savoie Mont Blanc, CNRS\\	F--73000 Chamb\'ery, France}
\email{philippe.laurencot@univ-smb.fr}
\author{Christoph Walker}
\address{Leibniz Universit\"at Hannover\\ Institut f\" ur Angewandte Mathematik \\ Welfengarten 1 \\ D--30167 Hannover\\ Germany}
\email{walker@ifam.uni-hannover.de}
\date{\today}

\begin{abstract}
Ranges of the real-valued parameters $\alpha$, $a$, $b$, and $m$ are identified for which the operator 
$$
\mathcal{A}_{\alpha}(a,b)f(x):=x^\alpha\left(f''(x)+\frac{a}{x}f'(x)+\frac{b}{x^2}f(x)\right), \quad x>0,
$$ 
generates an analytic semigroup in $L_1((0,\infty),x^m\mathrm{d}x)$.
\end{abstract}
%
\maketitle

\section{Introduction}\label{S1}

Let $(a,b,\alpha)\in \R^3$ and consider the second-order differential operator
\begin{equation}
	\mathcal{A}_\alpha(a,b) f(x):= x^\alpha\left(f''(x)+\frac{a}{x}f'(x)+\frac{b}{x^2}f(x)\right),\quad x>0, \label{Aaab}
\end{equation}
defined on the half-line~$(0,\infty)$ for $f\in \mathcal{D}'((0,\infty))$. We observe that $\mathcal{A}_\alpha(a,b)$ is degenerate for $\alpha>2$, while it is singular for $\alpha<2$. Using the weighted spaces
\begin{equation*}
    L_{p,m} := L_p((0,\infty),x^m\mathrm{d}x), \quad p\in [1,\infty], \quad m\in\mathbb{R},
\end{equation*}
\cite{MNS_JEE18, MNS_JDE22} discuss a range of the parameters $(\alpha,a,b)\in\mathbb{R}^3$, $p\in (1,\infty)$, and $m\in\mathbb{R}$ for which the (unbounded) linear operator $\big( \mathbb{B}_{p,m},\dom(\mathbb{B}_{p,m}) \big)$ defined by
\begin{equation*}
    \mathbb{B}_{p,m}f := \mathcal{A}_\alpha(a,b) f, \quad f\in \dom(\mathbb{B}_{p,m}),
\end{equation*}
with a certain domain $\dom(\mathbb{B}_{p,m})$ (which they identify) generates a bounded and positive analytic semigroup on $L_{p,m}$. Their analysis does not include the extreme cases $p=1$ and $p=\infty$ and the main purpose of this work is to obtain generation properties of $\mathcal{A}_{\alpha}(a,b)$ for the case $p=1$. This is not straightforward as can be inferred from \cite{Ama1983, Paz1983, Yag2010}, for instance. A different reason for the investigation of the case $p=1$ is the study of the well-posedness of diffusion-growth-fragmentation models of the form
\begin{equation*}
    \partial_t u = \partial_x^2(Du) + \partial_x(Gu) + \mathcal{F}[u], \quad (t,x)\in (0,\infty)^2,
\end{equation*}
where the diffusion coefficient $D\ge 0$ and the growth rate $G$ depend on the cluster size $x>0$, and $\mathcal{F}[f]$ is a fragmentation operator. See \cite{BaAr2006, BWWvB2004, FJMOS2003, Tch2024} and the references therein for more information about these models. Of particular importance in this context is the time evolution of the total mass of the population, which is nothing but the first moment of $u$ with respect to $x$. Weighted $L_1$-spaces are thus the natural functional framework for the study of well-posedness. See \cite{BaAr2006, BLL2020a} and the references they cite. The fragmentation operator does not involve derivatives and can be viewed as a lower order perturbation of the diffusion-growth operator. It is therefore possible to take the classical approach consisting in first establishing generation properties of the latter followed by an application of perturbation theory results in order to handle the full operator. This strategy was successfully implemented in \cite{LaWa2022} for the fragmentation equation with size diffusion. 

Returning to the operator $\mathcal{A}_\alpha(a,b)$, an established way to derive generation results in $L_1$-spaces relies on generation results in spaces of continuous functions combined with duality arguments. See, for instance, \cite{Ama1983, Paz1983, Yag2010}. The approach chosen here is different and similar to that in \cite{MNS_JEE18, MNS_JDE22}. It consists of two steps: in the first, the operator $\mathcal{A}_\alpha(a,b)$ is reduced by isometric transformations to the Bessel operator $\Gk$ defined by
\begin{equation}
    \Gk f(x):=x^{-\kappa}\Big(x^\kappa f'(x)\Big)'=f''(x)+\frac{\kappa}{x}f'(x),\quad x>0,\quad f\in \mathcal{D}'\big((0,\infty)\big), \label{gk}
\end{equation}
where $\kappa\in\left\{ 1 - \frac{2\sqrt{D}}{2-\alpha}, 1 + \frac{2\sqrt{D}}{2-\alpha} \right\}$, $\alpha\ne 2$, and $D:= (a-1)^2-4b\ge 0$. In the second step, the generation properties of $\Gk$ are studied in weighted $L_1$-spaces; this is done via a representation formula for the semigroup and the definition of the appropriate domain. 

We begin with generation results for $\Gk$ and, in order to simplify the notation, we set
\begin{equation*}
    X_m := L_{1,m} =L_1\big((0,\infty),x^m\mathrm{d}x\big),\quad m\in\mathbb{R},
\end{equation*} 
and
\begin{equation*}
	\Y_{I} := \bigcap_{r\in I} X_r,
\end{equation*}
for a non-empty interval $I\subset\mathbb{R}$.

\begin{thm}\label{T7}
	Let $\kappa<1$ and $\kappa-2<m\le 1$. Define the operator 
	\begin{equation*}
		A_{1,m}f := \Gk f, \qquad f\in \dom(A_{1,m}),
	\end{equation*} 
	where $\dom(A_{1,m})$ depends on the parameters $m$ and $\kappa$:
	\begin{itemize}
		\item [(c1)] if $\kappa-2<m\le \kappa$, then
		\begin{equation*}
			\dom(A_{1,m}) = \big\{ f \in \Y_{(\kappa-2,m]}\, \big |\, \Gk f\in X_m \big\} \,;
		\end{equation*}
		\item [(c2)] if $\kappa<m<1$, then
		\begin{equation*}
			\dom(A_{1,m}) = \big\{ f \in \Y_{[m-2,m]}\,\big | \, \Gk f\in X_m \big\}\,;
		\end{equation*}
		\item [(c3)] if $m=1$, then
		\begin{equation*}
			\dom(A_{1,1}) = \big\{ f \in \Y_{(-1,1]}\, \big |\, \Gk f\in X_1 
   \text{ and }\displaystyle{\lim_{x\to 0}f(x)=0}\big\}.
		\end{equation*}
	\end{itemize}
	Then $A_{1,m}$ generates a bounded and positive analytic semigroup $S_{1,m}$ on $X_m$ of angle $\pi/2$. In addition, $S_{1,m}$ is a semigroup of contractions on $X_m$ when $m\in (\kappa,1]$.
\end{thm}
	
The semigroup $S_{1,m}$ admits an explicit representation formula $S_{1,m}=\mathcal{S}|_{X_m}$ in terms of $\mathcal{S}$ defined in~\eqref{M1} below. The domain $\dom(A_{1,1})$ explicitly features a boundary condition, while such a condition is only implicitly present in $\dom(A_{1,m})$ for $m\in (\kappa,1)$. Indeed, if $m\in(\kappa,1)$ and $f\in D(A_{1,m})$, then it is a consequence of \Cref{C2} that
\begin{equation*}
	\lim_{x\to 0} x^{m-1} f(x) = 0.
\end{equation*}

In order to prove \Cref{T7}, we make use of the representation formula for the semigroup $S_{p,m}$ generated by $\Gk$ in $L_{p,m}$ in combination with the definition of a suitable domain that was derived in \cite{MNS_JDE22} for $p\in (1,\infty)$, $\kappa<1$, and $\kappa-1<\frac{m+1}{p}<2$. See~\eqref{M1} and \Cref{T0} below. $S_{p,m}$ coincides with the restriction of the integral operator $\mathcal{S}$ defined in~\eqref{M1} to the space $L_{p,m}$. The restriction of this integral operator to $X_m=L_{1,m}$ is thus a natural candidate for the semigroup generated by $\Gk$ on $X_m$. Using the representation formula, we prove that 
$$
S_{1,m}(z):=\mathcal{S}(z)|_{L_{1,m}}\in\mathcal{L}(X_m,X_{m-\theta}),\quad z\in \dot{\Sigma}_{\pi/2-\varepsilon},
$$
where $\varepsilon\in (0,\pi/2)$ and $\theta\in [0,m+2-\kappa)$ are arbitrary, provided $\kappa<1$ and $m\in (\kappa-2,1]$. Here and below, for $\alpha\in (0,\pi/2]$, we set 
\begin{equation*}
    \Sigma_{\alpha}:= \big\{ z\in\mathbb{C}\big |\,|\arg{z}|< \alpha \big\} \cup \{0\}, \quad \dot{\Sigma}_{\alpha}:= \Sigma_{\alpha}\setminus \{0\}.
\end{equation*}
We also prove that 
$$ 
S_{1,m}:\dot{\Sigma}_{\pi/2-\varepsilon}\to \mathcal{L}(X_m),\quad z\mapsto S_{1,m}(z)
$$
is continuous and, in fact, analytic. Thus, $S_{1,m}$ is an analytic semigroup on $X_m=L_{1,m}$ of angle $\pi/2$. The identification of its generator $A_{1,m}$ in $X_{m}$ requires two steps. In the first, a formula for its resolvent is obtained building on the results established in \cite{MNS_JDE22}. This leads to a complete characterization of the domain of the generator. It is necessary to handle the three cases $m\in (\kappa-2,\kappa]$, $m\in (\kappa,1)$, and $m=1$, separately. This results in slightly different definitions of the domain depending on the range of $m$. In the second step, we deal with the case when $m\in (\kappa,1]$. We show the dissipativity of the operator $\lambda - A_{1,m}$  in $X_m$ for $\lambda>0$. The contractivity of the semigroup $S_{1,m}$ is then a consequence of the Lumer-Phillips theorem. We point out that, in contrast to the analysis performed in \cite{MNS_JDE22}, \Cref{T7} includes the borderline case when $m=1$. In \cite{MNS_JDE22}, where $p\in (1,\infty)$, the borderline case $m=2p-1$ is excluded.

 As in \cite{MNS_JEE18, MNS_JDE22}, generation results for the general operator $\mathcal{A}_\alpha(a,b)$ can be derived from \Cref{T7}. We first deal with the singular case $\alpha<2$.
	
\begin{thm}\label{T5}
	Let $\alpha<2$ and take $(a,b)\in\R^2$ with $D:=(a-1)^2-4b\,\ge\,0$. Let $n\in (\nsb,\nsbp]$, where
	$$
	\nsb :=\frac{a-3-\sqrt{D}}{2}, \qquad \nsbpm:=\frac{1-2\alpha+a \pm \sqrt{D}}{2}.
	$$
	Define
	$$
	\mathbb{A}_n f:=\mathcal{A}_\alpha(a,b)f,\quad f\in \dom(\mathbb{A}_n),
	$$
	where $\dom(\mathbb{A}_n)$ depends on the parameters $\alpha$, $a$, $b$, and $n$ as follows:
	\begin{itemize}
		\item [(c1)] if $\nsb <n\le \nsbm$, then
		\begin{equation*}
			\dom(\mathbb{A}_n) = \big\{ f \in \Y_{(\nsb,n]}\, |\, \mathcal{A}_\alpha(a,b) f\in X_n \big\};
		\end{equation*}
		\item [(c2)] if $D>0$ and $\nsbm<n<\nsbp$, then
		\begin{equation*}
			\dom(\mathbb{A}_n) = \big\{ f \in \Y_{[n-2+\alpha,n]}\, |\, \mathcal{A}_\alpha(a,b) f\in X_n \big\};
		\end{equation*}
		\item [(c3)] if $n=\nsbp$, then
		\begin{equation*}
		\dom\big(\mathbb{A}_{\nsbp}\big) = \Big\{ f \in \Y_{(\nsbp-2+\alpha,\nsbp]}\, \Big |\,
		\mathcal{A}_\alpha(a,b) f\in X_{n^*_+},\:\displaystyle{\lim_{x\to 0} x^{(a+\sqrt{D}-1)/2}f(x) = 0}\Big\}.
		\end{equation*}
	\end{itemize}
	Then $\mathbb{A}_n$ generates a bounded and positive analytic semigroup on $X_n$ of angle $\pi/2$ which is a semigroup of contractions when $D>0$ and $n\in\left(\nsbm,\nsbp\right]$.
\end{thm}

\begin{rem}\label{R3}
For the specific operator $\mathcal{A}_\alpha\bigl(2\alpha,\alpha(\alpha-1)\bigr)$ with $\alpha<2$ explicitly given by
$$
\mathcal{A}_\alpha\bigl(2\alpha,\alpha(\alpha-1)\bigr)f(x) = x^\alpha\left( f''(x)+\frac{2\alpha}{x}  f'(x)+\frac{\alpha(\alpha-1)}{x^2}f(x)\right) = \left(x^\alpha f\right)'',\quad x>0,
$$
it holds that $D=1$ and that $(\nsb,\nsbm,\nsbp)=(\alpha-2,0,1)$.

Similarly, for the operator $\mathcal{A}_\alpha(\alpha,0)$ with $1\neq\alpha<2$ given by
$$
\mathcal{A}_\alpha(\alpha,0)f(x) = x^\alpha\left(f''(x)+\frac{\alpha}{x}  f'(x)\right) = \left(x^\alpha f'\right)' , \quad x>0,
$$
it holds that $D=|\alpha-1|$ and that 
$$
(\nsb,\nsbm,\nsbp)=\begin{cases}
 (-1,1-\alpha,0),&\alpha\in (1,2),\\
 (\alpha-2,0,1-\alpha),&\alpha\in (-\infty,1).
\end{cases}
$$
\end{rem}

A similar result is available in the degenerate case when $\alpha>2$.

\begin{thm}\label{T6}
	Let $\alpha>2$ and take $(a,b)\in\R^2$ with $D:=(a-1)^2-4b\ge 0$.  Let $n\in [\nsam,\nsa)$, where
	$$
	\nsapm :=\frac{1-2\alpha+a\pm \sqrt{D}}{2}, \qquad \nsa := \frac{a-3+\sqrt{D}}{2}.
	$$
Define
$$
\mathbb{A}_n f:=\mathcal{A}_\alpha(a,b)f,\quad f\in \dom(\mathbb{A}_n),
$$
where $\dom(\mathbb{A}_n)$ depends on the parameters $\alpha$, $a$, $b$, and $n$ as follows:
\begin{itemize}
	\item [(c1)] if $\nsap\le n<\nsa$, then
	\begin{equation*}
		\dom(\mathbb{A}_n) = \big\{ f \in \Y_{[n,\nsa)}\, \big |\,\mathcal{A}_\alpha(a,b) f\in X_n \big\};
	\end{equation*}
	\item [(c2)] if $D>0$ and $\nsam<n<\nsap$, then
	\begin{equation*}
		\dom(\mathbb{A}_n) = \big\{ f \in \Y_{[n,n-2+\alpha]}\, \big |\, \mathcal{A}_\alpha(a,b) f\in X_n \big\};
	\end{equation*}
	\item [(c3)] if $n=\nsam$, then
	\begin{multline*}
	\dom\big(\mathbb{A}_{\nsam}\big) = \Big\{ f \in \Y_{[\nsam,\nsam-2+\alpha)}\,\Big |\,
	\mathcal{A}_\alpha(a,b) f\in X_{\nsam},\:
	\displaystyle{\lim_{x\to \infty} x^{(a-\sqrt{D}-1)/2}f(x) = 0} \Big\}.
	\end{multline*}
\end{itemize}
Then $\mathbb{A}_n$ generates a bounded and positive analytic semigroup on $X_n$ of angle $\pi/2$ which is contractive when  $D>0$ and $n\in\left[\nsam,\nsap\right)$.
\end{thm}

As in \cite{MNS_JEE18, MNS_JDE22}, the proofs of \Cref{T5} and of \Cref{T6} rely on isometric transformations showing the similarity of $\mathcal{A}_\alpha(a,b)$ and $\Gk$ for appropriate values of $(\alpha,a,b)$ and~$\kappa$.

The rest of this paper is structured as follows. In \Cref{SP} we recall the results established in \cite{MNS_JDE22} and some useful properties enjoyed by the modified Bessel functions used throughout the paper. \Cref{S3} is devoted to the analysis of the operator $\Gk$ and to the proof of \Cref{T7}, from which \Cref{T5} and \Cref{T6} are deduced in \Cref{S4}. Generation results for singular or degenerate diffusion operators with absorption are established last. These build on the results for the diffusion part already established and on \cite{ArBa1993}. Auxiliary technical results are collected in the Appendices.

\section{Preliminaries}\label{SP}

The next theorem contains a summary of the results obtained in \cite{MNS_JDE22} for the homogeneous Dirichlet boundary value problem for the singular operator
$$
\Gk f(x) = f''(x) + \frac{\kappa}{x} f'(x), \quad x>0,
$$
in $L_{p,m}$ for $p\in (1,\infty)$. Its statement uses the modified Bessel functions $I_\nu$ and $K_\nu$. Their properties relevant to the paper at hand can be found in \cite[Section~9.6]{AbSt1972} and in \Cref{Bessel}. It also makes use of the integral operator defined by
\begin{subequations}\label{M1}
	\begin{equation}\label{m2}
		\big(\mathcal{S}(z) f\big)(x):=\int_0^\infty k_\kappa(z,x,r) f(r)\,\rd r
	\end{equation}
	with kernel given by
	\begin{equation}\label{m1}
		k_\kappa(z,x,r):=\frac{1}{2z}r^\kappa (xr)^{\frac{1-\kappa}{2}} \exp\left(-\frac{x^2+r^2}{4z}\right) I_{\frac{1-\kappa}{2}}\left(\frac{xr}{2z}\right)
	\end{equation}
	for $z\in\mathbb{C}$ with $\mathrm{Re}\,z>0$ and $(x,r)\in (0,\infty)^2$.
\end{subequations}

\begin{thm}[\cite{MNS_JDE22}]\label{T0}
Let $m\in\mathbb{R}$, $p\in (1,\infty)$, and $\kappa<1$ be such that $\kappa-1<\frac{m+1}{p}<2$. Then $S_{p,m} := \mathcal{S}|_{L_{p,m}}$ is a bounded and positive analytic semigroup on $L_{p,m}$ of angle $\pi/2$. If $B_{p,m}$ denotes its generator, then $B_{p,m}f=\Gk f$ for $f\in \dom(B_{p,m})$. Its resolvent is given by
\begin{subequations}\label{M2}
\begin{equation}\label{m3}
\big( (\lambda-B_{p,m})^{-1}f\big)(x) = \int_0^\infty G_\kappa(\lambda,x,r)\, f(r)\, r^\kappa\,\rd r,\quad x>0,
\end{equation}
for $\lambda\in \C\setminus (-\infty,0]$ and $f\in L_{p,m}$, where
\begin{equation}\label{m4}
G_\kappa(\lambda,x,r):=\left\{\begin{array}{ll} (xr)^\frac{1-\kappa}{2}\, I_\frac{1-\kappa}{2}(\sqrt{\lambda}x)\, K_\frac{1-\kappa}{2}(\sqrt{\lambda}r), & x\le r,\\[2mm]
(xr)^\frac{1-\kappa}{2}\, K_\frac{1-\kappa}{2}(\sqrt{\lambda}x)\, I_\frac{1-\kappa}{2}(\sqrt{\lambda}r), & x\ge r.
\end{array}\right.
\end{equation}
\end{subequations}
\end{thm}

\begin{proof}
This follows from \cite[Proposition~2.5, Theorem~2.6, and Proposition~3.2]{MNS_JDE22}.
\end{proof}

\begin{rem}\label{rem.dom}
		On the assumptions of \Cref{T0}, the domain of $B_{p,m}$ is characterized precisely in \cite[Proposition~3.5 and Corollary~3.6]{MNS_JDE22}.
\end{rem}

For later use, we recall the symmetries
\begin{equation}
	x^{\kappa} k_\kappa(z,x,r) = r^{\kappa} k_\kappa(z,r,x) , \quad G_\kappa(z,x,r) = G_\kappa(z,r,x), \label{sym}
\end{equation}
valid for $(x,r)\in (0,\infty)^2$ and where $z\in\mathbb{C}$ with $\mathrm{Re}\,z>0$.

The next lemma collects several properties of the modified Bessel functions $I_\nu$ and $K_\nu$ which can be found in \cite[Sections~9.6 and~9.7]{AbSt1972}.

\begin{lem}\label{Bessel}
Let $\nu>-1$. The modified Bessel functions $(I_\nu,K_\nu)$ form a fundamental system of the second-order linear differential equation
\begin{equation}
	z^2 v''(z)+ z v'(z)-(z^2+\nu^2)v(z)=0,\quad \mathrm{Re}\,z>0. \label{eq_mbessel}
\end{equation}
Moreover,
\begin{align}
	I_\nu'(z) & = I_{\nu+1}(z) + \frac{\nu}{z} I_\nu(z), \quad \mathrm{Re}\, z>0, \label{edI} \\
	K_\nu'(z) & = - K_{\nu+1}(z) + \frac{\nu}{z} K_\nu(z), \quad \mathrm{Re}\, z>0. \label{edK}
\end{align}

It holds that
	\begin{subequations}\label{mbp}
		\begin{align}
	I_\nu(z)\sim & 
   \frac{1}{\Gamma(\nu+1)} \left( \frac{z}{2} \right)^\nu,\quad\nu>-1,
   \label{mbpa}
   \\[2pt]
			K_\nu(z)  \sim & \begin{cases}\frac{\Gamma(|\nu|)}{2} \left( \frac{z}{2} \right)^{-\nu},& \nu\ne 0,\\ - \ln{z} ,& \nu=0,\end{cases}\label{mbpb}
		\end{align}
		as $z\to 0$ and
		\begin{equation}
			I_\nu(z) \sim \frac{1}{\sqrt{2\pi}}\,\frac{e^{z}}{|z|^{\frac{1}{2}}}, \qquad K_\nu(z) \sim \sqrt{\frac{\pi}{2}} \,\frac{e^{-z}}{|z|^{\frac{1}{2}}} \label{mbpc}
		\end{equation}
    as $z\to\infty$.
	\end{subequations}
	There are positive constants $C_I$ and $C_K$ such that
	\begin{subequations}\label{i20i21}
		\begin{align}
			0&<I_\nu(x)\le C_I\big[ x^\nu\, \chi_{(0,1]}(x)+x^{-1/2}\,e^x \,\chi_{[1,\infty)}(x)\big],\quad x>0,\label{i20}
			\\0&< K_\nu(x)\le C_K\big[ x^{-\nu}\, \chi_{(0,1]}(x)+x^{-1/2}\,e^{-x} \,\chi_{[1,\infty)}(x)\big],\quad x>0.\label{i21}
		\end{align}
	\end{subequations}
Finally, for $\ve>0$ there is $C=C(\nu,\ve)>0$ such that
\begin{equation}\label{In}
\vert I_\nu(z)\vert\le C\,\big(1\wedge \vert z\vert\big)^{\nu+1/2}\, \frac{e^{\mathrm{Re}\, z}}{\vert z\vert ^{\frac{1}{2}}},\quad z\in \dot{\Sigma}_{\pi/2-\ve}.
\end{equation}
\end{lem}

\section{The Semigroup and its Generator\texorpdfstring{ in $X_m$}{}}\label{S3}

We first prove that $S_{1,m}:=\mathcal{S}|_{X_m}$, see~\eqref{M1}, is an analytic semigroup on $X_m$ whenever $m\in (\kappa-2,1]$. 

\begin{prop}\label{P11}
Let $\kappa <1$, $\kappa-2<m\le 1$, and $S_{1,m}=\mathcal{S}|_{X_m}$. Then  $S_{1,m}$ is a bounded and positive analytic semigroup on $X_m$ of angle $\pi/2$. Moreover, for each $\ve>0$ and  $\theta\ge 0$ with $\theta+\kappa-2<m\le 1$, there is $c(\ve,\theta)>0$ such that
\begin{equation}\label{3.1}
\big\| S_{1,m}(z)\big\|_{\ml(X_m,X_{m-\theta})}\le c(\ve,\theta) \vert z\vert^{-\theta/2},\quad z\in  \dot\Sigma_{\pi/2-\ve}.
\end{equation}
\end{prop}

\begin{proof}
We divide the proof into several steps.
\begin{itemize}[leftmargin=*]
\item[\bf (i)] Let $\ve>0$ be arbitrary but fixed. Then, due to \cite[Proposition~2.8]{MNS_JDE22}, there are constants $C:=C(\ve)>0$ and $s:=s(\ve)>0$ such that, for $z\in  \dot\Sigma_{\pi/2-\ve}$ and $(x,r)\in (0,\infty)^2$, one has
\begin{equation}\label{k11}
0\le |k_\kappa (z,x,r)| \le \frac{C}{\sqrt{\vert z\vert}} \left(\frac{x}{\sqrt{\vert z\vert}}\wedge 1\right)^{1-\kappa}\left (\frac{r}{\sqrt{\vert z\vert}}\wedge 1\right)\,\exp\left(-\frac{\vert x-r\vert^2}{s \vert z\vert}\right).
\end{equation}
Therefore, given $\theta\ge 0$ with $\theta+\kappa-2<m\le 1$, we infer from~\eqref{m2} and Proposition~\ref{P11x} that
\begin{equation*}
\big\| S_{1,m}(z)f \big\|_{X_{m-\theta}} \le \int_0^\infty  \int_0^\infty  k_\kappa(z,x,r) \, \vert f(r)\vert \,\rd r\, x^{m-\theta}\,\rd x \le c\,\vert z\vert^{-\theta/2}\, \| f \|_{X_m}
\end{equation*} 
for $f\in X_m$ and  $z\in  \dot\Sigma_{\pi/2-\ve}$, which establishes estimate~\eqref{3.1}.

\item[\bf (ii)] Next we turn to the strong continuity of $S_{1,m}$ and prove it first for a function $f\in X_m\cap L_\infty((0,\infty))$ with compact support in $[a,b]$, where $0<a<1<b$, in which case we have that
\begin{equation}\label{note}
	\big (S_{1,m}(z) f \big)(x) = \int_a^b k_\kappa(z,x,r)\, f(r)\,\rd r, \qquad x\in (0,\infty).
\end{equation}
Consider $z\in \dot\Sigma_{\pi/2-\ve}$ satisfying
\begin{equation}
	|z| \le 1 \wedge \frac{a^2}{4}. \label{z02}
\end{equation}
Then, for $x>b$ and $r\in [a,b]$, one has $x>r>\sqrt{|z|}$ and thus, by~\eqref{k11}, 
\begin{align*}
	J_\infty(z) & := \int_b^\infty \left| \int_a^b k_\kappa(z,x,r)\, f(r)\,\rd r \right| \, x^m\,\rd x \\
	& \le \frac{C}{\sqrt{|z|}} \int_b^\infty \int_a^b \exp\left\{ - \frac{(x-r)^2}{s|z|} \right\} |f(r)| x^m\, \rd r\rd x \\
	& = C \int_a^b |f(r)| \int_{(b-r)/\sqrt{|z|}}^\infty \left( r + y \sqrt{|z|} \right)^m e^{-y^2/s}\, \rd y\rd r.
\end{align*}
Either $m\le 0$ and $\left( r + y \sqrt{|z|} \right)^m \le r^m$, so that
\begin{subequations}\label{z03}
\begin{equation}
	J_\infty(z) \le C \int_a^b |f(r)| r^m \int_{(b-r)/\sqrt{|z|}}^\infty e^{-y^2/s}\, \rd y\rd r. \label{z03a}
\end{equation}
Or $m\in (0,1]$ and 
\begin{equation*}
	\left( r + y \sqrt{|z|} \right)^m \le r^m + y^m \sqrt{|z|}^m \le r^m \left( 1 +y^m \right)
\end{equation*}
by~\eqref{z02}, which implies that
\begin{align}
	J_\infty(z) & \le C \int_a^b |f(r)| r^m \int_{(b-r)/\sqrt{|z|}}^\infty \big( 1 + y^m \big) e^{-y^2/s}\, \rd y\rd r \nonumber\\
	& \le C \int_a^b |f(r)| r^m \int_{(b-r)/\sqrt{|z|}}^\infty \left( 1 + \sup_{x>0}\left\{ x^m e^{-x^2/2s} \right\} \right) e^{-y^2/2s}\, \rd y\rd r \nonumber\\
	& \le C  \int_a^b |f(r)| r^m \int_{(b-r)/\sqrt{|z|}}^\infty e^{-y^2/2s}\, \rd y\rd r. \label{z03b}
\end{align}
\end{subequations}
Introducing
\begin{equation*}
	\Lambda_\infty(\xi) := \int_a^b |f(r)| r^m \int_{(b-r)/\xi}^\infty e^{-y^2/2s}\, \rd y\rd r, \qquad \xi>0,
\end{equation*}
we deduce from~\eqref{z03} that
\begin{equation}
	J_\infty(z) \le C \Lambda_\infty\big(\sqrt{|z|}\big). \label{z04}
\end{equation}
Now, for $\xi\in \big(0,(b-a)^2\big)$,
\begin{align*}
	\Lambda_\infty(\xi) & = \int_{a}^{b-\sqrt{\xi}} |f(r)| r^m \int_{(b-r)/\xi}^\infty e^{-y^2/2s}\, \rd y\rd r \\
	& \qquad + \int_{b-\sqrt{\xi}}^b |f(r)| r^m \int_{(b-r)/\xi}^\infty e^{-y^2/2s}\, \rd y\rd r \\
	& \le \int_{a}^{b-\sqrt{\xi}} |f(r)| r^m \int_{1/\sqrt{\xi}}^\infty e^{-y^2/2s}\, \rd y\rd r \\
	& \qquad + \int_{b-\sqrt{\xi}}^b |f(r)| r^m \int_0^\infty e^{-y^2/2s}\, \rd y\rd r \\
	& \le \|f\|_{X_m} \int_{1/\sqrt{\xi}}^\infty e^{-y^2/2s}\, \rd y + C \int_{b-\sqrt{\xi}}^b |f(r)| r^m\,\rd r,
\end{align*}
and we infer from the integrability properties of $f\in X_m$ that
\begin{equation*}
	\lim_{\xi\to 0} \Lambda_\infty(\xi) = 0.
\end{equation*}
Consequently,
\begin{equation}
	\lim_{z\to 0} J_\infty(z) = 0. \label{z06}
\end{equation}

Next, for $r\in [a,b]$ and $z\in \dot\Sigma_{\pi/2-\ve}$ satisfying~\eqref{z02}, one has $r>\sqrt{|z|}$ as well as $\sqrt{|z|}\in [0,a/2]$, so that, by~\eqref{k11},
\begin{align*}
	J_0(z) &  :=  \int_0^a \left| \int_a^b k_\kappa(z,x,r)\, f(r)\,\rd r \right| \, x^m\,\rd x \\
	& \le \frac{C}{\sqrt{|z|}} \int_0^a \int_a^b \left( \frac{x}{\sqrt{|z|}} \wedge 1 \right)^{1-\kappa} \exp\left\{ - \frac{(x-r)^2}{s |z|} \right\} |f(r)| x^m\, \rd r\rd x \\
	& = \frac{C}{|z|^{(2-\kappa)/2}} \int_0^{\sqrt{|z|}} \int_a^b x^{1-\kappa+m} \exp\left\{ - \frac{(x-r)^2}{s |z|} \right\} |f(r)|\, \rd r\rd x \\
	& \qquad + \frac{C}{\sqrt{|z|}} \int_{\sqrt{|z|}}^a \int_a^b x^{m} \exp\left\{ - \frac{(x-r)^2}{s |z|} \right\} |f(r)|\, \rd r\rd x \\
	& \le  \frac{C}{|z|^{(2-\kappa)/2}} \int_a^b |f(r)| \int_0^{\sqrt{|z|}} x^{1-\kappa+m} \exp\left\{ - \frac{r^2}{4s |z|} \right\}\, \rd x\rd r \\
	& \qquad + \frac{C}{\sqrt{|z|}} \int_a^b |f(r)|  \int_{\sqrt{|z|}}^a x^{m} \exp\left\{ - \frac{(x-r)^2}{s |z|} \right\}\, \rd x\rd r.
\end{align*}
Since $m+2-\kappa>0$, we further obtain
\begin{align*}
	J_0(z) & \le  C|z|^{m/2} \int_a^b |f(r)| \exp\left\{ - \frac{r^2}{4s |z|} \right\}\, \rd r \\
	& \qquad + C \int_a^b |f(r)| \int_{(r-a)/\sqrt{|z|}}^{(r-\sqrt{|z|})/\sqrt{|z|}} \big( r - y \sqrt{|z|} \big)^m e^{-y^2/s}\, \rd y\rd r \\
	& = C \int_a^b r^m |f(r)| \left( \frac{\sqrt{|z|}}{r} \right)^m \exp\left\{ - \frac{r^2}{4s |z|} \right\}\, \rd r \\
	& \qquad + C \int_a^b r^m |f(r)| \int_{(r-a)/\sqrt{|z|}}^{(r-\sqrt{|z|})/\sqrt{|z|}} \left( \frac{r - y \sqrt{|z|}}{r} \right)^m e^{-y^2/s}\, \rd y\rd r.
\end{align*}
Either $m\in [0,1]$ and, since $r>a>\sqrt{|z|}$ by~\eqref{z02},
\begin{equation*}
	\left( \frac{\sqrt{|z|}}{r} \right)^m \le 1, \qquad \left( \frac{r - y \sqrt{|z|}}{r} \right)^m \le 1,
\end{equation*}
so that
\begin{subequations}\label{z07}
\begin{equation}
	\begin{split}
		J_0(z) & \le C \int_a^b r^m |f(r)|  \exp\left\{ - \frac{r^2}{4s |z|} \right\}\, \rd r \\
		& \qquad + C \int_a^b r^m |f(r)| \int_{(r-a)/\sqrt{|z|}}^{(r-\sqrt{|z|})/\sqrt{|z|}} e^{-y^2/s}\, \rd y\rd r.
	\end{split}\label{z07a}
\end{equation}
Or $m<0$ and
\begin{equation*}
	\left( \frac{\sqrt{|z|}}{r} \right)^m \exp\left\{ - \frac{r^2}{4s |z|} \right\} \le \sup_{y>0}\left\{ y^{-m} e^{-y^2/8s}\right\} \exp\left\{ - \frac{r^2}{8s |z|} \right\},
\end{equation*}
while, for $y\sqrt{|z|}\in\Big(r-a\, ,\,r - \sqrt{|z|}\Big)$,
\begin{align*}
	 \left( \frac{r - y \sqrt{|z|}}{r} \right)^m & =  \left( \frac{r - y \sqrt{|z|} + y \sqrt{|z|}}{r-y \sqrt{|z|}} \right)^{-m} \le 2^{-m} \left[ 1 + \left( \frac{y\sqrt{|z|}}{r-y \sqrt{|z|}} \right)^{-m} \right] \\
	 & \le 2^{-m} \left[ 1 + \big( y\sqrt{|z|} \big)^{-m} |z|^{m/2} \right] = 2^{-m} \big( 1 + y^{-m} \big).
\end{align*}
Consequently,
\begin{align}
	J_0(z) & \le C \int_a^b r^m |f(r)| \exp\left\{ - \frac{r^2}{8s |z|} \right\}\, \rd r \nonumber\\
	& \qquad + C \int_a^b r^m |f(r)| \int_{(r-a)/\sqrt{|z|}}^{(r-\sqrt{|z|})/\sqrt{|z|}} \big( 1 + y^{-m} \big) e^{-y^2/s}\, \rd y\rd r \nonumber\\
	& \le C \int_a^b r^m |f(r)| \exp\left\{ - \frac{r^2}{8s |z|} \right\}\, \rd r \nonumber\\
	& \qquad + C \int_a^b r^m |f(r)| \int_{(r-a)/\sqrt{|z|}}^{(r-\sqrt{|z|})/\sqrt{|z|}} e^{-y^2/2s}\, \rd y\rd r. \label{z07b}
\end{align}
\end{subequations}
Introducing
\begin{align*}
	\Lambda_0(\xi) & := \int_a^b r^m |f(r)| \exp\left\{ - \frac{r^2}{8s \xi^2} \right\}\, \rd r + \int_a^b r^m |f(r)| \int_{(r-a)/\xi}^{(r-\xi)/\xi} e^{-y^2/2s}\, \rd y\rd r
\end{align*}
for $\xi>0$, we have shown that
\begin{equation}
	J_0(z) \le C \Lambda_0(\sqrt{|z|}). \label{z08}
\end{equation}
Now, it follows from the integrability properties of $f\in X_m$ and from Lebesgue's dominated convergence theorem that
\begin{equation*}
	\lim_{\xi\to 0} \int_a^b r^m |f(r)| \exp\left\{ - \frac{r^2}{8s \xi^2} \right\}\, \rd r = 0.
\end{equation*}
Furthermore, for $\xi\in \big(0, (b-a)^2\big)$,
\begin{align*}
	\int_a^b r^m |f(r)| \int_{(r-a)/\xi}^{(r-\xi)/\xi} e^{-y^2/2s}\, \rd y\rd r & \le \int_a^{a+\sqrt{\xi}} r^m |f(r)|\, \rd r \int_0^\infty e^{-y^2/2s}\, \rd y \\
	& \qquad + \int_{a+\sqrt{\xi}}^b r^m |f(r)|\, \rd r \int_{\xi^{-1/2}}^\infty e^{-y^2/2s}\,\rd y 
\end{align*}
and, using again the integrability properties of $f\in X_m$ and Lebesgue's dominated convergence theorem, we conclude that
\begin{equation*}
	\lim_{\xi\to 0} \int_a^b r^m |f(r)| \int_{(r-a)/\xi}^{(r-\xi)/\xi} e^{-y^2/2s}\, \rd y\rd r = 0.
\end{equation*}
Consequently, 
\begin{equation*}
	\lim_{\xi\to 0} \Lambda_0(\xi) = 0,
\end{equation*}
which implies, together with~\eqref{z08}, that
\begin{equation}
	\lim_{z\to 0} J_0(z) = 0. \label{z09}
\end{equation}

We now choose $p\in (1,\infty)$ wih $\kappa-1<\frac{m+1}{p}<2$. Since $f\in X_m\cap L_\infty((0,\infty))\subset L_{p,m}$ and using~\eqref{note}, we obtain that
\begin{align*}
	\big\| S_{1,m}(z)f - f\big\|_{X_m} & = \int_0^a \left|\big(S_{1,m}(z)f\big)(x)\right| x^m\, \rd x + \int_a^b \left|\big(S_{1,m}(z)f\big)(x) - f(x)\right| x^m\, \rd x \\
	& \qquad + \int_b^\infty \left|\big(S_{1,m}(z)f\big)(x)\right| x^m\, \rd x \\
	& \le J_0(z) + J_\infty(z) + \int_a^b \left|\big(S_{p,m}(z)f\big)(x) - f(x)\right| x^m\, \rd x \\
	& \le J_0(z) + J_\infty(z) + C\,\left\|S_{p,m}(z)f - f\right\|_{L_{p,m}}\,.
\end{align*}
As strong continuity of $S_{p,m}$ in $L_{p,m}$ follows from Theorem~\ref{T0}, we infer from~\eqref{z06}, \eqref{z09}, and the above inequality that
\begin{equation}
	\lim_{z\to 0} \| S_{1,m}(z)f - f\|_{X_m} = 0 \label{z10}
\end{equation}
for any $f\in X_m\cap L_\infty((0,\infty))$ with compact support in $(0,\infty)$. It only remains to combine the already established boundedness property~\eqref{3.1} (with $\theta=0$) with a density argument in order to conclude that~\eqref{z10} is valid for all $f\in X_m$.

\item[{\bf (iii)}] As for the analyticity of~$S_{1,m}$ we note from \Cref{Bessel} that, for $z\in \C$ with $\mathrm{Re}\, z>0$ and $(x,r)\in (0,\infty)^2$, 
\begin{align*}
z \partial_zk_\kappa(z,x,r)&=-\frac{1}{2z}r^\kappa (xr)^{\frac{1-\kappa}{2}} \exp\left(-\frac{x^2+r^2}{4z}\right) I_{\frac{1-\kappa}{2}}\left(\frac{xr}{2z}\right)\\&\phantom{=}+
\frac{1}{2z}r^\kappa (xr)^{\frac{1-\kappa}{2}}\frac{x^2+r^2}{4z}\exp\left(-\frac{x^2+r^2}{4z}\right) I_{\frac{1-\kappa}{2}}\left(\frac{xr}{2z}\right)\\&\phantom{=}-
\frac{1}{2z}r^\kappa (xr)^{\frac{1-\kappa}{2}}\frac{xr}{2z}\exp\left(-\frac{x^2+r^2}{4z}\right) I'_{\frac{1-\kappa}{2}}\left(\frac{xr}{2z}\right)\\
&=k_\kappa(z,x,r)\left[ \frac{x^2+r^2}{4z}-\frac{3-\kappa}{2} \right]\\
&\phantom{=}-\frac{xr}{4z^2}r^\kappa (xr)^{\frac{1-\kappa}{2}} \exp\left(-\frac{x^2+r^2}{4z}\right)I_{\frac{3-\kappa}{2}}\left(\frac{xr}{2z}\right) \\
& = k_\kappa(z,x,r)\left[ \frac{x^2+r^2}{4z}-\frac{3-\kappa}{2}  - \frac{xr}{2z} \frac{I_{\frac{3-\kappa}{2}}}{I_{\frac{1-\kappa}{2}}} \left( \frac{xr}{2z} \right) \right] \\
& = k_\kappa(z,x,r)\left[ \frac{(x-r)^2}{4z}-\frac{3-\kappa}{2}  + \frac{xr}{2z} \left( 1 - \frac{I_{\frac{3-\kappa}{2}}}{I_{\frac{1-\kappa}{2}}} \left( \frac{xr}{2z} \right) \right) \right].
\end{align*}
We now argue as in the proof of \cite[Proposition~2.9]{MNS_JDE22} and use the validity of
\begin{equation*}
    \left| 1 - \frac{I_{\frac{3-\kappa}{2}}}{I_{\frac{1-\kappa}{2}}}(\xi) \right| \le C(\varepsilon) \left( 1 \wedge \frac{1}{|\xi|} \right), \quad \xi\in \dot{\Sigma}_{\pi/2-\ve}\,,
\end{equation*}
along with~\eqref{k11} and \cite[Lemma~10.1]{MNS_JDE22}, to obtain that, for $z\in \dot{\Sigma}_{\pi/2-\ve}$ and $(x,r)\in (0,\infty)^2$,
\begin{align*}
   \big| z \partial_z k_\kappa(z,x,r) \big| & \le |k_\kappa(z,x,r)| \left[ \left| \frac{3-\kappa}{2} \right| + s \exp\left(\frac{(x-r)^2}{4s \vert z\vert}\right) + C \left( 1 \wedge \frac{xr}{2|z|} \right) \right] \\
   & \le C |k_\kappa(z,x,r)| \left[ 1 + \left( 1 \wedge \frac{x}{\sqrt{|z|}} \right) \left( 1 \wedge \frac{r}{\sqrt{|z|}} \right) \right] \exp\left(\frac{(x-r)^2}{4s |z|}\right) \\
   & \le \frac{C}{\sqrt{|z|}} \left( 1 \wedge \frac{x}{\sqrt{|z|}} \right)^{1-\kappa} \left( 1 \wedge \frac{r}{\sqrt{|z|}} \right)\,\exp\left(-\frac{(x-r)^2}{2s |z|}\right) \\
   & \phantom{=} + \frac{C}{\sqrt{|z|}} \left( 1 \wedge \frac{x}{\sqrt{|z|}} \right)^{2-\kappa} \left( 1 \wedge \frac{r}{\sqrt{|z|}} \right)^2\,\exp\left(-\frac{(x-r)^2}{2s |z|}\right) \\
   & \le C \Big[ q_{2s,\kappa-1,-1}(|z|,x,r) + q_{2s,\kappa-2,-2}(|z|,x,r) \Big],
\end{align*}
noticing that $q_{s,\alpha,\beta}$ is defined in~\eqref{k11x}. Since 
\begin{equation*}
    \kappa-2 < \kappa - 1 < m+1 \le 2 < 3,
\end{equation*}
we infer from \Cref{P11x} (with $(\alpha,\beta,\theta)\in \{(\kappa-2,3,0),(\kappa-1,2,0)\}$) that, for $f\in X_m$,
\begin{equation*}
    S_{1,m}f:\dot\Sigma_{\pi/2-\ve}\rightarrow X_m,\qquad z\mapsto S_{1,m}(z)f
\end{equation*}
is differentiable and that
\begin{equation}\label{key}
\left\| z\frac{\rd }{\rd z} S_{1,m}(z)f\right\|_{X_m}=  \left\| z \int_0^\infty \partial_z k_\kappa(z,\cdot,r) f(r)\,\rd r\right\|_{X_m}\le c\|f\|_{X_m},\quad z\in \dot\Sigma_{\pi/2-\ve}\,.
\end{equation}
Therefore, the semigroup $S_{1,m}$ is analytic with angle $\pi/2$ due to~\cite[Theorem~II.4.6]{EnNa2000}, as~\eqref{key} shows that we can apply this result to any ray in the sector $\Sigma_{\pi/2-\ve}$. This concludes the proof. 
\end{itemize}
\end{proof}

For $\kappa<1$ and $\kappa-2<m\le 1$, we denote the generator of the semigroup $S_{1,m}$ on $X_m$ by $A_{1,m}$ and turn our attention to identifying the domain of $A_{1,m}$ in the form stated in \Cref{T7}. This proves to be quite involved. We begin by verifying that $A_{1,m}$ has the expected differential form and by establishing an integral representation of its resolvent.

\begin{lem}\label{L100}
	Let $\kappa<1$ and $\kappa-2<m\le 1$. If $f\in \dom(A_{1,m})$ and $\lambda>0$, then $g:= \lambda f - A_{1,m} f$ belongs to $X_m$ and
	\begin{align}
		& A_{1,m}f = \Gk f \;\;\text{ in }\;\; \mathcal{D}'\bigl((0,\infty)\bigr), \label{x001} \\
		& f(x) = \int_0^\infty G_\kappa(\lambda,x,r)\, g(r)\, r^\kappa\,\rd r, \qquad x>0, \label{x002}
	\end{align}
    where $G_\kappa$ is defined in~\eqref{m4}.
\end{lem}

\begin{proof}
	Let $\vartheta\in \mathcal{D}\bigl((0,\infty)\bigr)$ and fix $p>1$ such that $\kappa-1 < (m+1)/p<2$. Since $S_{1,m}$ is bounded, it holds that $\lambda\in \rho(A_{1,m})$ and hence
	\begin{equation}\label{Ay}
		f = \int_0^\infty e^{-\lambda t}\, S_{1,m}(t)g\,\rd t.
	\end{equation}
	Thus, we infer from~\eqref{M1}, \eqref{sym},  Fubini's theorem, and Theorem~\ref{T0} that
	\begin{align*}
		\int_0^\infty f(x)\, \vartheta(x)\, x^\kappa\,\rd x & = \int_0^\infty x^\kappa\, \vartheta(x)\, \int_0^\infty e^{-\lambda t}\, \big[ S_{1,m}(t)g \big](x)\,\rd t\,\rd x \\
		& = \int_0^\infty e^{-\lambda t}  \int_0^\infty x^\kappa\, \vartheta(x)\, \big[ S_{1,m}(t)g \big](x)\,\rd x\,\rd t \\
		& = \int_0^\infty e^{-\lambda t}  \int_0^\infty \int_0^\infty x^\kappa\, k_\kappa(\lambda,x,r)\, g(r)\, \vartheta(x)\, \rd r\,\rd x\,\rd t \\
		& = \int_0^\infty e^{-\lambda t}  \int_0^\infty \int_0^\infty r^\kappa\, k_\kappa(\lambda,r,x)\, g(r)\, \vartheta(x)\, \rd r\,\rd x\,\rd t \\
		& = \int_0^\infty e^{-\lambda t}  \int_0^\infty r^\kappa\, g(r)\, \big[ S_{p,m}(t)\vartheta \big](r)\,\rd r\,\rd t \\
		& = \int_0^\infty r^\kappa\, g(r)\, \int_0^\infty e^{-\lambda t}\, \big[ S_{p,m}(t)\vartheta \big](r)\,\rd t\,\rd r \\
		& = \int_0^\infty r^\kappa\, g(r)\, \big[ \big( \lambda-B_{p,m}\big)^{-1}\vartheta \big](r)\,\rd r \\
		& = \int_0^\infty r^\kappa\, g(r)\, \int_0^\infty G_\kappa(\lambda,r,x)\, \vartheta(x)\, x^\kappa \, \rd x\,\rd r \\
		& = \int_0^\infty x^\kappa\, \vartheta(x)\, \int_0^\infty G_\kappa(\lambda,x,r)\, g(r)\, r^\kappa \, \rd r\,\rd x. 
	\end{align*}
	As this identity is valid for any $\vartheta\in \mathcal{D}\bigl((0,\infty)\bigr)$, we arrive at~\eqref{x002}.
	
	Similarly, since $\big(\lambda-B_{p,m}\big)\vartheta$ belongs to $\mathcal{D}\bigl((0,\infty)\bigr)$, it follows from~\eqref{M1} and~\eqref{sym}, along with Fubini's theorem, that
	\begin{align*}
		& \int_0^\infty x^\kappa\, f(x)\, \big[ \big(\lambda-B_{p,m}\big)\vartheta \big](x)\,\rd x \\
		& \qquad = \int_0^\infty e^{-\lambda t}  \int_0^\infty r^\kappa\, g(r)\, \big[ S_{p,m}(t)\big(\lambda-B_{p,m}\big)\vartheta \big](r)\,\rd r\,\rd t \\
		& \qquad = \int_0^\infty r^\kappa\, g(r) \int_0^\infty e^{-\lambda t} \, \big[ S_{p,m}(t)\big(\lambda-B_{p,m}\big)\vartheta \big](r)\,\rd t\,\rd r \\
		& \qquad = \int_0^\infty r^\kappa\, g(r)\, \big[ \big( \lambda - B_{p,m})^{-1}\big(\lambda-B_{p,m}\big)\vartheta \big](r)\,\rd r \\
		& \qquad = \int_0^\infty r^\kappa\, g(r) \,\vartheta(r)\,\rd r = \int_0^\infty r^\kappa\, \vartheta(r)\, \big[ (\lambda - A_{1,m})f \big](r)\,\rd r.
	\end{align*}
	Consequently, by \Cref{T0}, it holds that
	\begin{equation*}
		\int_0^\infty x^\kappa\, f(x)\, \Gk\vartheta(x)\,\rd x = \int_0^\infty x^\kappa\, \vartheta(x)\, A_{1,m}f(x)\,\rd x,
	\end{equation*}
	from which we deduce~\eqref{x001}.
\end{proof}

Next we identify $\dom(A_{1,m})$. The following lemma is needed.

\begin{lem}\label{L101}
	Let $\kappa<1$. Let $f\in X_n$ with 
	\begin{equation}
		\text{ either }\; n\le -1 \;\text{ or }\; n>-1 \;\text{ and }\; \lim_{x\to 0} \frac{1}{x} \int_0^x f(z)\,\rd z = 0 \label{x005},
	\end{equation}
	and such that $f - \Gk f = 0$ in $(0,\infty)$. Then $f\equiv 0$.
\end{lem}

\begin{proof}
	We argue as in the proof of \cite[Proposition~3.3]{MNS_JDE22}. Set $\nu=(1-\kappa)/2$ and notice that $\nu>0$ owing to $\kappa<1$. Then $v(x):=x^{-\nu} f(x)$, $x>0$, solves the modified Bessel equation
	\begin{equation}
		x^2 v''(x)+x v'(x)-(x^2+\nu^2)v(x)=0,\qquad x>0. \label{x003}
	\end{equation}
	Since $(I_\nu,K_\nu)$ forms a fundamental system for~\eqref{x003}, as recalled in \Cref{Bessel}, the Wronskian of which is given by $W(I_\nu,K_\nu)(x)=-x^{-1}$ (see \cite[9.6.15]{AbSt1972}), there is $(c_1, c_2)\in \R^2$ such that
	\begin{equation*}
		f(x)=c_1x^\nu I_\nu(x)+c_2 x^\nu K_\nu(x),\qquad x>0.
	\end{equation*}
	Since $I_{\nu}(x)$ is exponentially increasing as $x\to\infty$ by~\eqref{mbpc}, $f\in X_n$ implies that $c_1=0$. Next, $x^\nu K_\nu(x)$ has a positive limit as $x\to 0$ by~\eqref{mbpb}, so that \eqref{x005} entails that $c_2=0$. This completes the proof.
\end{proof}

We are now in a position to identify $\dom(A_{1,m})$.

\begin{prop}\label{P102}
	Let $\kappa<1$ and $\kappa-2<m\le 1$. Then $A_{1,m}f=\Gk f\in X_m$ for $f\in \dom(A_{1,m})$ and,
	\begin{itemize}
		\item [(c1)] if $\kappa-2<m\le \kappa$, then
		\begin{equation*}
			\dom(A_{1,m}) = \Big\{ f \in \Y_{(\kappa-2,m]}\,\Big |\,
				\Gk f\in X_m
			 \Big\};
		\end{equation*}
		\item [(c2)] if $\kappa<m<1$, then
		\begin{equation*}
			\dom(A_{1,m}) = \Big\{ f \in \Y_{[m-2,m]}\,\Big |\, \Gk f\in X_m \Big\};
		\end{equation*}
		\item [(c3)] if $m=1$, then
		\begin{equation*}
			\dom(A_{1,1}) = \Big\{ f \in \Y_{(-1,1]}\,\Big |\,\Gk f\in X_1,\:\displaystyle{\lim_{x\to 0} f(x) = 0}
			\Big\}.
		\end{equation*}
	\end{itemize}
\end{prop}

\begin{proof}
Consider $f\in \dom(A_{1,m})$. Then, $A_{1,m}f = \Gk f\in X_m$ by~\eqref{x001}.  We set $g=f-A_{1,m}f$ and infer from~\eqref{Ay} and \Cref{P11}, as in the proof of \cite[Proposition~3.3]{MNS_JDE22}, that, for $\theta\in [0,m+2-\kappa)$,
\begin{equation*}
	\|f\|_{X_{m-\theta}} \le \int_0^\infty e^{-t}\, \|S_{1,m}(t)g\|_{X_{m-\theta}}\,\rd t \le c(1,\theta) \|g\|_{X_m} \int_0^\infty t^{-\theta/2}\, e^{-t}\, \rd t, 
\end{equation*}
from which we deduce that $f\in X_{m-\theta}$ for all $0\le \theta < \min\{m+2-\kappa,2\}$. Equivalently,
\begin{equation*}
	f\in \Y_{\bigl((m-2)\vee(\kappa-2),m\bigr]}
\end{equation*}
so that
\begin{equation}
        \dom(A_{1,m}) \subset D :=   \Big\{ f \in \Y_{ \bigl((\kappa-2)\vee(m-2),m\bigr]}\,\Big |\, \Gk f\in X_m \Big\}. \label{x004}
\end{equation}

\medskip

\noindent\textbf{Case~1} [$m<1$]: In this case $-1\in \big((m-2)\vee(\kappa-2), m \big]$ and thus $\dom(A_{1,m})\subset X_{-1}$. Given $f\in D$, there is $u\in \dom(A_{1,m})$ such that $f-\Gk f = u-A_{1,m}u=u-\Gk u$. Since $f-u\in X_{-1}$, we readily deduce from \Cref{L101} that $f=u$ and thus that $f\in \dom(A_{1,m})$. Recalling~\eqref{x004}, we conclude that $\dom(A_{1,m}) = D$ and we have, in particular, proved \Cref{P102}~(c1).\\[.15cm]
\noindent\textbf{Case~2} [$m\in (\kappa,1)$]. It remains to show that $\dom(A_{1,m})=D\subset X_{m-2} $. To this end, we argue as in \cite[Proposition~3.3]{MNS_JDE22} and consider $f\in \dom(A_{1,m})$. Since $f\in X_m$ with $\Gk f\in X_m$, we deduce from \Cref{L1} that $f'\in X_{m-1}$ with
\begin{equation*}
	f'(x)=-x^{-\kappa}\int_x^\infty z^\kappa\, \Gk f(z)\,\rd z,\quad x>0.
\end{equation*}
In particular, $f'\in L_1\bigl((0,1)\bigr)$ and there is $b\in\mathbb{R}$ such that
\begin{align*}
	f(x) & = b - \int_0^x y^{-\kappa}\int_y^\infty z^\kappa\, \Gk f(z)\,\rd z\rd y \\
	& = b - \frac{1}{1-\kappa} \int_0^x z\, \Gk f(z)\,\rd z - \frac{x^{1-\kappa}}{1-\kappa} \int_x^\infty z^\kappa \, \Gk f(z)\,\rd z, \quad x>0.
\end{align*}
For $x\in (0,1)$, 
\begin{equation*}
	\left| \int_0^x z\, \Gk f(z)\,\rd z \right| \le \int_0^x z^m\, z^{1-m}\, |\Gk f(z)|\,\rd z \le x^{1-m} \|\Gk f\|_{X_m},
\end{equation*}
and
\begin{equation*}
	\left| x^{1-\kappa} \int_x^\infty z^\kappa \, \Gk f(z)\,\rd z\right| \le x^{1-\kappa} \int_x^\infty z^m\, z^{\kappa-m} \, |\Gk f(z)|\,\rd z \le x^{1-m} \|\Gk f\|_{X_m},
\end{equation*}
so that $\lim_{x\to 0} f(x) = b$. However, $f\in X_{-1}$ by~\eqref{x004} and thus it must hold that $b=0$. Therefore
\begin{equation*}
	f(x) = - \frac{1}{1-\kappa} \int_0^x z\, \Gk f(z)\,\rd z - \frac{x^{1-\kappa}}{1-\kappa} \int_x^\infty z^\kappa \, \Gk f(z)\,\rd z, \quad x>0,
\end{equation*}
and
\begin{align*}
	(1-\kappa) \int_0^1& x^{m-2}\, |f(x)|\,\rd x 
  \\
  &\le \int_0^1 x^{m-2} \int_0^x z\, |\Gk f(z)|\,\rd z\rd x+ \int_0^1 x^{m-1-\kappa} \int_x^1 z^\kappa\, |\Gk f(z)|\,\rd z\rd x\\ 
  &\qquad + \int_0^1 x^{m-1-\kappa} \int_1^\infty z^\kappa\, |\Gk f(z)|\,\rd z\rd x \\
 &\le \int_0^1 z\, \frac{z^{m-1}-1}{1-m}\, |\Gk f(z)|\,\rd z + \frac{1}{m-\kappa} \int_0^1 z^m |\Gk f(z)|\,\rd z\\
 &\qquad+ \frac{1}{m-\kappa} \int_1^\infty z^\kappa\, |\Gk f(z)|\,\rd z \\
 &\le \frac{1-\kappa}{(1-m)(m-\kappa)} \|\Gk f\|_{X_m}.
\end{align*}
Therefore $f\in X_{m-2}$ and we have shown that $\dom(A_{1,m})\subset X_{m-2}$, which completes the proof of~(c2).\\[.15cm]
\noindent\textbf{Case~3} [$m = 1$]: Bearing \eqref{x004} in mind, we now study the behavior of functions in $\dom(A_{1,1})$ as $x\to 0$. Consider $f\in \dom(A_{1,1})$ and set $g=f-A_{1,1}f$. Thanks to~\eqref{x002}, $f$ admits the following representation formula
\begin{equation*}
	f(x) = \int_0^\infty G_\kappa(1,x,r)\, g(r)\, r^\kappa\,\rd r,\quad x>0.
\end{equation*}
In particular, given $x\in (0,1)$, it follows from~\eqref{m4} and~\eqref{i20i21} with $\nu=(1-\kappa)/2>0$ that
\begin{align*}
	\big| f(x) \big| & \le \int_0^x (xr)^\nu\, K_\nu(x)\, I_\nu(r)\, |g(r)|\, r^{\kappa}\ \rd r + \int_x^\infty (xr)^\nu\, I_\nu(x)\, K_\nu(r)\, |g(r)|\, r^{\kappa}\ \rd r\\
	& \le C_I C_K \int_0^x r\, |g(r)|\, \rd r + C_I C_K x^{1-\kappa} \int_x^1 r^{\kappa}\, |g(r)|\, \rd r \\
	& \qquad + C_I C_K x^{1-\kappa} \int_1^\infty r^{\kappa/2}\, e^{-r}\, |g(r)|\, \rd r \\
	& \le C_I C_K \int_0^x r\, |g(r)|\, \rd r + C_I C_K x^{1-\kappa} \int_x^1 r^{\kappa}\, |g(r)|\, \rd r \\
	& \qquad + C_I C_K x^{1-\kappa} \sup_{r\ge 1}\left\{ r^{(\kappa-2)/2} e^{-r} \right\} \|g\|_{X_1}.
\end{align*}
Since $\kappa<1$, 
\begin{equation*}
	\left( \frac{x}{r} \right)^{1-\kappa} r |g(r)| \chi_{(x,1)}(r) \le r |g(r)|,\quad (x,r)\in (0,1)^2,
\end{equation*}
and
\begin{equation*}
	\lim_{x\to 0} \left( \frac{x}{r} \right)^{1-\kappa} r |g(r)| \chi_{(x,1)}(r) = 0\ \text { for a.e. }r\in (0,1).
\end{equation*}
Lebesgue's dominated convergence theorem implies that
\begin{equation*}
	\lim_{x\to 0} x^{1-\kappa} \int_x^1 r^{\kappa}\, |g(r)|\, \rd r = 0.
\end{equation*}
Noting that the two other terms on the right-hand side of the upper bound for $\big| f(x) \big|$ converge to zero as $x\to 0$ due to $\kappa<1$ and to $g\in X_1$, we conclude that 
\begin{equation*}
	\lim_{x\to 0}  f(x) = 0.
\end{equation*}
Recalling~\eqref{x004}, we arrive at
\begin{equation*} 
	\dom(A_{1,1}) \subset D_1 :=  \Big\{ f \in \Y_{(-1,1]}\,\Big |\, \Gk f\in X_1,\: \displaystyle{\lim_{x\to 0} f(x) = 0}\Big\}.
\end{equation*}
We next proceed as in the proof of~(c1) to derive the identity $\dom(A_{1,1}) = D_1$ using \Cref{L101} with $n=0$.
\end{proof}

While $\dom(A_{1,1})$ features explicitly a homogeneous Dirichlet boundary condition at $x=0$, a similar boundary condition is implicitly included in $\dom(A_{1,m})$ for $m\in (\kappa,1)$ as we now show.
	
\begin{cor}\label{C2}
	Let $m\in(\kappa,1]$. If $f\in \dom(A_{1,m})$, then
	\begin{equation*}
		\lim_{x\to 0} x^{m-1} f(x) = 0.
	\end{equation*}
\end{cor}

\begin{proof}
	As the case $m=1$ is settled, we consider $m\in (\kappa,1)$ and let $f\in \dom(A_{1,m})$. Since $f\in X_m\cap X_{m-2}$ and $\Gk f\in X_m$ by \Cref{P102}~(c2), we infer from \Cref{C1} that
	\begin{equation*}
		\ell := \lim_{x\to 0} x^{m-1} f(x)
	\end{equation*}
	exists and this property is compatible with $f\in X_{m-2}$ only when $\ell=0$.
\end{proof}

We are left with proving that $S_{1,m}$ is a semigroup of contractions on $X_m$ for $m\in (\kappa,1]$. In spite of the explicit representation of $S_{1,m}=\mathcal{S}\vert_{X_m}$ given in~\eqref{m2}, this property does not seem to follow directly due to the modified Bessel functions involved. We therefore prove contractivity by establishing the dissipativity of $A_{1,m}$ in $X_m$.
	
\begin{lem}\label{L2}
Let $m\in(\kappa,1]$. If $\lambda>0$ and $f\in \dom(A_{1,m})$, then
\begin{equation}\label{diss}
\|(\lambda-A_{1,m})f\|_{X_m}\ge \lambda \|f\|_{X_m}+(1-m)(m-\kappa)\|f\|_{X_{m-2}}.
\end{equation}
\end{lem}

\begin{proof}
 Let $\lambda>0$ and $f\in \dom(A_{1,m})$. Given $\ve>0$ define $\beta_\ve\in W_{2,loc}^2(\R)$ by
\begin{align}\label{sign}
\beta_\ve(r):=\left\{\begin{array}{ll}
r-\ve/2, & r\ge \ve,\\[2mm]
r^2/(2\ve), & r\in (-\ve,\ve),\\[2mm]
-r-\ve/2,& r\le -\ve,
\end{array}\right.
\end{align} 
and note that 
\begin{equation}\label{beta}
\begin{split}
0 \le \beta_\ve(r)& \le \vert r\vert,\quad \vert \beta_\ve'(r)\vert\le 1,\quad r\in\R, \\
\big| \beta_\varepsilon(r) - |r| \big| & \le \frac{\varepsilon}{2}, \quad \lim_{\varepsilon\to 0} \beta_\varepsilon'(r) = \mathrm{sign}(r),\quad r\in\R.
\end{split}
\end{equation}
Integration by parts and using~\eqref{beta} yield
\begin{align*}
\|(\lambda  - A_{1,m}) f\|_{X_m} & \ge \int_0^\infty x^m\, \beta_\ve'(f(x)) \left[\lambda  f(x)-f''(x)-\frac{\kappa}{x}f'(x)\right]\, \mathrm{d}x\\
&= \lambda \int_0^\infty x^m\, \beta_\ve'(f(x)) \,  f(x)\, \mathrm{d}x- \Big[x^m\,\beta_\ve'(f(x))\, f'(x)\Big]_{x=0}^{x=\infty}\\
&\quad + \int_0^\infty x^m\, \beta_\ve''(f(x)) \,  \vert f'(x)\vert^2\, \mathrm{d}x\\
&\quad +(1-m)(m-\kappa)\int_0^\infty x^{m-2}\, \beta_\ve(f(x))\, \mathrm{d}x\\
&\quad + \Big[ (m-\kappa)\, x^{m-1}\, \beta_\ve(f(x))\Big]_{x=0}^{x=\infty}.
	\end{align*}		
As all boundary terms vanish in view of \eqref{f1x}, \eqref{f2x}, \eqref{beta}, and owing to the definition of $\dom(A_{1,m})$, we infer from the convexity of $\beta_\varepsilon$ that
\begin{align*}
    \|(\lambda  - A_{1,m}) f\|_{X_m} & \ge  \lambda \int_0^\infty x^m\, \beta_\ve'(f(x)) \,  f(x)\, \mathrm{d}x \\
    & \qquad +(1-m)(m-\kappa)\int_0^\infty x^{m-2}\, \beta_\ve(f(x))\, \mathrm{d}x.
	\end{align*}	
Since $f\in \dom(A_{1,m})$ and the second term on the right-hand side vanishes when $m=1$, we may pass to the limit as $\varepsilon\to 0$ with the help of Lebesgue's dominated convergence theorem and of~\eqref{beta} to obtain
\begin{equation*}
\|(\lambda-A_{1,m})f\|_{X_m}\ge \lambda \|f\|_{X_m}+(1-m)(m-\kappa)\|f\|_{X_{m-2}},
\end{equation*}
as claimed.
\end{proof}

\begin{cor}\label{P1}
Let $m\in(\kappa,1]$. Then $S_{1,m}$ is a contraction semigroup on $X_m$.
\end{cor}

\begin{proof}
Since $A_{1,m}$ is the generator of the semigroup $S_{1,m}$ on $X_m$ and is a dissipative operator in $X_m$ by~\eqref{diss}, the claim is an immediate consequence of the Lumer–Phillips theorem \cite[Chapter~1, Theorem~4.3]{Paz1983}.
\end{proof}

\Cref{T7} is now a consequence of \Cref{P11}, \Cref{P102}, and \Cref{P1}.

\section{General Operators}\label{S4}

We consider now more general operators of the form
$$
\mathcal{A}_\alpha(a,b) f(x):= x^\alpha\left(f''(x)+\frac{a}{x}f'(x)+\frac{b}{x^2}f(x)\right),\quad x>0,
$$
for $f\in \mathcal{D}'((0,\infty))$ with given $(a,b,\alpha)\in \R^3$ and observe that $\mathcal{A}_\alpha(a,b)$ is singular for $\alpha<2$ and degenerate for $\alpha>2$. Following~\cite{MNS_JEE18} we transform such an operator to an operator of the form $\Gk$ (with a suitable $\kappa$) and apply the result of \Cref{T7}. To this end, the following isometries play an instrumental role (see \cite[Proposition~3.1 and Proposition~3.2]{MNS_JEE18} and \cite[Lemma~4.1]{MNS_JDE22}):

\begin{lem}\label{L14}
Given $n\in \R$, $l\in\R$,  and $\beta\in \R\setminus\{-1\}$ define
$$
(\Tk_\beta  u)(x):=\vert \beta+1\vert \, u\big(x^{1+\beta}\big), \quad (\mathcal{M}_l  u)(x):=x^l \, u(x) 
$$
for $x>0$ and $u\in X_n$. Then $\Tk_\beta: X_n\to X_{n(1+\beta)+\beta}$ is an isometric isomorphism with 
$\Tk_\beta^{-1}=\Tk_{-\beta/(1+\beta)}$ and $\mathcal{M}_l: X_{n}\to X_{n-l}$ is an isometric isomorphism with $\mathcal{M}_l^{-1}=\mathcal{M}_{-l}$.
Moreover,
$$
\mathcal{M}_l \Tk_\beta = \Tk_\beta \mathcal{M}_{l/(1+\beta)},
$$
and, for $u\in W_{1,loc}^2((0,\infty))$, it holds that
\begin{align*}
\partial_x (\Tk_\beta  u) &= (1+\beta) \Tk_\beta \mathcal{M}_{\beta/(1+\beta)} \partial_x u,\\
\partial_x^2 (\Tk_\beta  u) &= (1+\beta)^2\Tk_\beta \mathcal{M}_{2\beta/(1+\beta)} \partial_x^2 u + \beta (1+\beta) \Tk_\beta \mathcal{M}_{(\beta-1)/(1+\beta)} \partial_x u,\\
\partial_x (\mathcal{M}_l  u) &=l\mathcal{M}_{l-1} u + \mathcal{M}_l \partial_x u,\\
\partial_x^2 (\mathcal{M}_l  u) &=\mathcal{M}_l \partial_x^2 u +2l\mathcal{M}_{l-1} \partial_x u + l(l-1)\mathcal{M}_{l-2}u.
\end{align*}
\vspace{2mm}
Finally, we have that 
\begin{equation*}
    w \big(\mathcal{M}_l v\big) = \mathcal{M}_l(vw), \quad  w \big(\Tk_\beta v\big) = |1+\beta| \Tk_\beta\big(v \Tk_{-\beta/(1+\beta)}w\big)
\end{equation*}
for any two real-valued functions $v$ and $w$ defined on $(0,\infty)$, for $l\in\mathbb{R}$, and for $\beta\in\mathbb{R}\setminus\{-1\}$.
\end{lem}

We transform the operator $\mathcal{A}_\alpha(a,b)$ to a singular operator $\mathcal{G}_\kappa$ for a suitable $\kappa$ depending on $\alpha$, $a$, and $b$.

\begin{lem}\label{L16}
Assume that $\alpha\ne 2$ and 
\begin{subequations}\label{k}
\begin{equation}\label{k1}
D:=(a-1)^2-4b\ge 0 .
\end{equation}
Introducing
\begin{equation}\label{k2}
l_\pm:=-\frac{a-1}{2-\alpha}\pm\frac{\sqrt{D}}{2-\alpha}
\end{equation}
and
\begin{equation}\label{k4}
\kappa_\pm:=2l_\pm+\frac{2a-\alpha}{2-\alpha}=1\pm\frac{2\sqrt{D}}{2-\alpha},
\end{equation}
we find
\begin{equation}
\mathcal{A}_\alpha(a,b) =\frac{(2-\alpha)^2}{4}\, \big(\Tk_{-\alpha/2}\mathcal{M}_{l_\pm}\big)\, \mathcal{G}_{\kappa_\pm}\,\big(\Tk_{-\alpha/2}\mathcal{M}_{l_\pm}\big)^{-1}\label{k3} .
\end{equation}
\end{subequations}
\end{lem}

\begin{proof} 
Observing that
$$
\mathcal{A}_\alpha(a,b) = \mathcal{M}_\alpha \partial_x^2 + a \mathcal{M}_{\alpha-1} \partial_x + b \mathcal{M}_{\alpha-2},
$$
it follows from Lemma~\ref{L14} that
\begin{align*}
\mathcal{A}_\alpha(a,b) \Tk_{-\alpha/2}  &= \mathcal{M}_\alpha\partial_x^2\Tk_{-\alpha/2}  + a \mathcal{M}_{\alpha-1} \partial_x \Tk_{-\alpha/2}  + b\, \mathcal{M}_{\alpha-2} \Tk_{-\alpha/2}   \\
&= \mathcal{M}_\alpha\,\Tk_{-\alpha/2} \left[ \frac{(2-\alpha)^2}{4}\,\mathcal{M}_{-2\alpha/(2-\alpha)}\, \partial_x^2  -\frac{\alpha(2-\alpha)}{4}\, \mathcal{M}_{-(2+\alpha)/(2-\alpha)}\,\partial_x  \right]\\
&\quad + \frac{a(2-\alpha)}{2} \mathcal{M}_{\alpha-1}\, \Tk_{-\alpha/2} \mathcal{M}_{-\alpha/(2-\alpha)}\, \partial_x   + b\, \mathcal{M}_{\alpha-2} \Tk_{-\alpha/2}  \\
&= \frac{(2-\alpha)^2}{4}\, \Tk_{-\alpha/2} \partial_x^2   - \frac{\alpha(2-\alpha)}{4}\, \Tk_{-\alpha/2}\mathcal{M}_{-1}\,\partial_x  \\
&\quad + \frac{a(2-\alpha)}{2}\, \Tk_{-\alpha/2} \mathcal{M}_{-1}\, \partial_x   +b\, \Tk_{-\alpha/2} \mathcal{M}_{-2}\,   .
\end{align*}
Therefore,
\begin{subequations}\label{R}
\begin{equation}\label{R1}
 \Tk_{-\alpha/2}^{-1} \,\mathcal{A}_\alpha(a,b)\, \Tk_{-\alpha/2} 
=\frac{(2-\alpha)^2}{4}\,  \mathcal{A}_0(A,B)
\end{equation}
with
\begin{equation}\label{R2}
A := \frac{2a-\alpha}{2-\alpha}, \quad B:= \frac{4b}{(2-\alpha)^2}.
\end{equation}
\end{subequations}
Next, we use Lemma~\ref{L14} to obtain
\begin{align*}
 \mathcal{A}_0(A,B)\,\mathcal{M}_l  &= \partial_x^2 \mathcal{M}_l  + A\, \mathcal{M}_{-1}\,\partial_x \mathcal{M}_l  + B\, \mathcal{M}_{-2} \mathcal{M}_l  \\
&= \mathcal{M}_l\partial_x^2  + 2l\, \mathcal{M}_{l-1} \partial_x  + l(l-1)\,\mathcal{M}_{l-2} \\
&\quad + lA\, \mathcal{M}_{l-2}  + A\mathcal{M}_{l-1} \partial_x  +B\, \mathcal{M}_{l-2} ,
\end{align*}
and thus
\begin{equation}\label{R4}
\begin{split}
\mathcal{M}_{-l}\,\mathcal{A}_0(A,B)\,\mathcal{M}_l   &=    \partial_x^2\,  + \left(2l+\frac{2a-\alpha}{2-\alpha}\right)\, \mathcal{M}_{-1}\,\partial_x  \\
& \qquad + \left(l^2+\frac{2a-2}{2-\alpha}l +\frac{4b}{(2-\alpha)^2} \right)\, \mathcal{M}_{-2}.
\end{split}
\end{equation}
Combining~\eqref{R} and~\eqref{R4} yields
\begin{align*}
\big(\Tk_{-\alpha/2}&\mathcal{M}_l\big)^{-1}\mathcal{A}_\alpha(a,b)\big(\Tk_{-\alpha/2}\mathcal{M}_l\big)\\
&=\frac{(2-\alpha)^2}{4}\left\{\partial_x^2  +\left(2l+\frac{2a-\alpha}{2-\alpha}\right)\mathcal{M}_{-1}\partial_x +\left(l^2+\frac{2a-2}{2-\alpha}l+\frac{4b}{(2-\alpha)^2}\right)\mathcal{M}_{-2}\right\}.
\end{align*}
Choosing $l$ suitably we can get rid of the zero order term. More precisely, the non-negativity~\eqref{k1} of $D$ ensures that $l_\pm$ in~\eqref{k2} are well-defined and taking $l\in \{l_\pm\}$ in the above formula gives
\begin{align*}
\big(\Tk_{-\alpha/2}\mathcal{M}_{l_\pm}\big)^{-1}\mathcal{A}_\alpha(a,b)\big(\Tk_{-\alpha/2}\mathcal{M}_{l_\pm}\big)&=\frac{(2-\alpha)^2}{4}\left\{\partial_x^2  +\left(2l_\pm+\frac{2a-\alpha}{2-\alpha}\right)x^{-1}\partial_x\right\}\,;
\end{align*}
that is,
\begin{align*}
\mathcal{A}_\alpha(a,b)
&=\frac{(2-\alpha)^2}{4}\, \big(\Tk_{-\alpha/2}\mathcal{M}_{l_\pm}\big)\, \mathcal{G}_{\kappa_\pm}\,\big(\Tk_{-\alpha/2}\mathcal{M}_{l_\pm}\big)^{-1} 
\end{align*}
with $\kappa_\pm$ given by~\eqref{k4}, as claimed.
\end{proof}

Noticing from \Cref{L14} that, for any $n\in\R$, $\big(\Tk_{-\alpha/2}\mathcal{M}_{l_\pm}\big)^{-1}$ is an isometric isomorphism from $X_n$ onto $X_{m_\pm(n)}$ with
\begin{equation}
m_\pm(n):=\frac{2n+\alpha}{2-\alpha}+l_\pm=\frac{2n+\alpha-a+1}{2-\alpha}\pm \frac{\sqrt{D}}{2-\alpha}, \label{mpmn}
\end{equation}
we can make use of the results from \Cref{S3} and \Cref{L16} in order to obtain a positive analytic semigroup on $X_n$ for $\mathcal{A}_\alpha(a,b)$ with suitably chosen values of $\alpha$, $a$, $b$, and $n$. We are now ready for the proofs of \Cref{{T5}} and \Cref{T6}.

\begin{proof}[Proof of \Cref{{T5}} (Singular case)]
Fix $n\in \left(\nsb,\nsbp\right]$ and set $l:=l_-$, $\kappa:=\kappa_-$, and $m:=m_-(n)$. Then, since $D>0$ and $\alpha<2$, we have
$$
\kappa-2=-1-\frac{2\sqrt{D}}{2-\alpha} <m = \frac{2n+\alpha-a+1}{2-\alpha}- \frac{\sqrt{D}}{2-\alpha}\le 1.
$$
It then follows from \Cref{T7} that $A_{1,m}$ with domain $\dom(A_{1,m})$ defined therein generates a bounded and positive analytic semigroup on $X_{m}$ of angle $\pi/2$.
Since
\begin{equation}\label{fff}
\Tk_{-\alpha/2}\mathcal{M}_{l}: X_r \longrightarrow X_{[(r-l)(2-\alpha)-\alpha]/2}
\end{equation} 
is an isometric isomorphism for all $r\in\mathbb{R}$ by \Cref{L14} and
\begin{equation}
	\big[ \big(\Tk_{-\alpha/2}\mathcal{M}_{l}\big)^{-1}f \big](x) = \frac{2}{|2-\alpha|} x^{-l} f\left( x^{2/(2-\alpha)} \right), \quad x>0, \label{x006}
\end{equation} 
it follows from \eqref{k3} that
$$
\big(\Tk_{-\alpha/2}\mathcal{M}_{l} \big) \big(\dom(A_{1,m})\big)= \dom(\mathbb{A}_n).
$$
Consequently, since $\big(\Tk_{-\alpha/2}\mathcal{M}_{l}\big)^{-1}:X_n\longrightarrow X_{m}$ is an isometric isomorphism,  $\mathbb{A}_n$ generates a bounded and positive analytic semigroup on $X_n$ of angle $\pi/2$ which is a semigroup of contractions on $X_n$ for $n\in \big(\nsbm,\nsbp\big]$.
\end{proof}

\begin{proof}[Proof of \Cref{T6} (Degenerate case)]
Fix $n\in \left[\nsam,\nsa\right)$ and set $l:=l_+$, $\kappa:=\kappa_+$, and $m:=m_+(n)$. Then, since $D>0$ and $\alpha>2$, we have
$$
\kappa=1+\frac{2\sqrt{D}}{2-\alpha} <m= \frac{2n+\alpha-a+1}{2-\alpha}+ \frac{\sqrt{D}}{2-\alpha}\le 1.
$$
Thanks to \Cref{T7}, we may argue as in the proof of \Cref{T5}, paying special attention to the negativity of $2-\alpha$ in~\eqref{fff} and~\eqref{x006}, to complete the proof of \Cref{T6}. 
\end{proof}

\section{The Absorption Semigroup}\label{S3B}

Let
\begin{equation}\label{a1}
	\omega: (0,\infty)\to \mathbb{R} \;\text{ be a non-negative measurable function}
\end{equation}
and define
\begin{subequations}\label{Al}
	\begin{equation}
		A_{1,m}^\omega f:=A_{1,m} f-\omega f,\quad f\in \dom(A_{1,m}^\omega),
	\end{equation}
	with domain
	\begin{equation}  
		\dom(A_{1,m}^\omega):=\big\{f\in \dom(A_{1,m})\,\big |\,  \omega f\in X_m\big\}.
	\end{equation}
\end{subequations}

We shall show that $A_{1,m}^\omega$ generates a positive analytic semigroup of contractions on $X_m$ for $m\in (\kappa,1]$.

\begin{lem}\label{LL1}
	Assume that $\omega$ satisfies~\eqref{a1} and let $m\in (\kappa,1]$. If it holds that $(f,g)\in X_m^2$ and $(f_n)_{n\ge 1}$ is a sequence in $\dom(A_{1,m})$ such that 
	\begin{equation}\label{eq1}
		\lim_{n\to\infty} \|f_n - f\|_{X_m} = \lim_{n\to\infty} \|g_n-g\|_{X_m} = 0
	\end{equation}
	with $g_n:=-A_{1,m} f_n +(\omega\wedge n) f_n$, then $f\in \dom(A_{1,m}^\omega)$ with $A_{1,m}^\omega f=-g$.
\end{lem}

\begin{proof}
	\noindent\textbf{(i)} We first prove that $f\in X_{m-2}$ when $m<1$ and $\omega f\in X_{m}$. To this end, let $\ve\in (0,1)$ and $\beta_\ve$ be defined as in~\eqref{sign}. Then, as in the proof of Lemma~\ref{L2} with $\lambda=0$ we have
	\begin{align*}
		\|g_n\|_{X_m} & \ge \int_0^\infty x^m\, \beta_\ve'(f_n(x))\, g_n(x)\, \mathrm{d}x \\
		&= \int_0^\infty x^m\, \beta_\ve'(f_n(x))\, \left(-f_n''(x)-\frac{\kappa}{x} f_n'(x)\right)\, \mathrm{d}x \\
		& \qquad +\int_0^\infty x^m\, \beta_\ve'(f_n(x))\, (\omega(x)\wedge n)f_n(x)\, \mathrm{d}x\\
		& \ge (1-m)(m-\kappa) \int_0^\infty x^{m-2}\, \beta_\ve(f_n(x))\, \mathrm{d}x \\
		& \qquad +\int_0^\infty x^{m}\, (\omega(x)\wedge n)\, \beta_\ve'(f_n(x))\,f_n(x)\, \mathrm{d}x.
	\end{align*}
	Letting $\ve\to 0$
	we derive
	\begin{align*}
		(1-m)(m-\kappa) \int_0^\infty x^{m-2}\, \vert f_n(x)\vert \, \mathrm{d}x+ \int_0^\infty x^{m}\, (\omega(x)\wedge n)\, |f_n(x)|\, \mathrm{d}x\le c_0 ,
	\end{align*}
	where $c_0:=\sup_{n\in\mathbb{N} }\|g_n\|_{X_m}<\infty$.
	For $N\ge 1$ fixed and $n\ge N$, we obtain
	\begin{align*}
		(1-m)(m-\kappa) \int_0^\infty x^{m-2}\, \vert f_n(x)\vert \, \mathrm{d}x+ \int_0^\infty x^{m}\, (\omega(x)\wedge N) \, \vert f_n(x)\vert\, \mathrm{d}x\le c_0.
	\end{align*}
	Thus, letting 
	first $n\to\infty$ using \eqref{eq1} and then $N\to \infty$ using Fatou's lemma, we find
	\begin{align*}
		(1-m)(m-\kappa) \int_0^\infty x^{m-2}\, \vert f(x)\vert \, \mathrm{d}x+ \int_0^\infty x^{m}\, \omega(x)\, \vert f(x)\vert\, \mathrm{d}x\le c_0.
	\end{align*}
	Consequently, $f\in X_{m-2}$ when $m<1$ and $\omega f\in X_{m}$.  
	
	\medskip
	
	\noindent\textbf{(ii)} Next, we show that $\Bigl((\omega\wedge n) f_n\Bigr)_{n\ge 1}$ converges to $\omega f$ in  $X_{m}$. Let $\vartheta\in C^\infty\bigl((0,\infty)\bigr)$ with $0\le \vartheta\le 1$ satisfy $\vartheta(x)=1$ for $x>2$ and $\vartheta(x)=0$ for $x\in (0,1)$. Set $\vartheta_R(x) := \vartheta(x/R)$ for $x>0$ and $R>1$. Then, for $\ve\in (0,1)$, integration by parts and the convexity of $\beta_\ve$ yield
	\begin{align*}
		\int_0^\infty& x^m\,\vartheta_R(x)\, \beta_\ve'(f_n(x))\, \left(-f_n''(x)-\frac{\kappa}{x} f_n'(x)\right)\, \mathrm{d}x \\
		&= \int_0^\infty \left[ (m-\kappa) x^{m-1}\,\vartheta_R(x)+x^m\,\vartheta_R'(x)\right]\, \big(\beta_\ve(f_n(x))\big)'\, \mathrm{d}x\\
		&\quad + \int_0^\infty x^m\,\vartheta_R(x)\, \beta_\ve''(f_n(x))\, \vert f_n'(x)\vert^2\, \mathrm{d}x\\
		&\ge -\int_0^\infty \left[(2m-\kappa)\, x^{m-1}\,\vartheta_R'(x)+x^m\,\vartheta_R''(x)\right]\, \beta_\ve(f_n(x))\, \mathrm{d}x\\
		&\quad + (m-\kappa)(1-m)\int_0^\infty x^{m-2}\,\vartheta_R(x)\, \beta_\ve(f_n(x))\, \mathrm{d}x,
	\end{align*}
	where the boundary terms vanish due to the definition of $\vartheta_R$ and Lemma~\ref{L1}. Thanks to the non-negativity of the last term and to
	\begin{align*}
		\left\vert(2m-\kappa)\, x^{m-1}\,\vartheta_R'(x)+x^m\,\vartheta_R''(x)\right\vert \le (2m-\kappa+1)\,\frac{x^m}{R^2}\,\|\vartheta''\|_\infty, 
	\end{align*}
	we deduce from~\eqref{beta} that
	\begin{equation*}\label{o1}
		\begin{split}
			\int_0^\infty x^m\,\vartheta_R(x)\, &\beta_\ve'(f_n(x))\, \left(-f_n''(x)-\frac{\kappa}{x} f_n'(x)\right)\, \mathrm{d}x \ge - (2m-\kappa+1)\,\frac{1}{R^2}\,\|\vartheta''\|_\infty \,\|f_n\|_{X_m}.
		\end{split}
	\end{equation*}
	Consequently, by the definition of $g_n$ we obtain
	\begin{align*}
		\int_0^\infty x^{m}\,&\vartheta_R(x)\, \beta_\ve'(f_n(x))\, \big(\omega(x)\wedge n\big)\, f_n(x)\, \mathrm{d}x\\
		&\le \int_0^\infty x^m\,\vartheta_R(x)\, \beta_\ve'(f_n(x))\, g_n(x)\, \mathrm{d}x + (2m-\kappa+1)\,\frac{1}{R^2}\,\|\vartheta''\|_\infty \,\|f_n\|_{X_m}\\
		&\le \int_R^\infty x^m\, \vert g_n(x)\vert \, \mathrm{d}x
		+ (2m-\kappa+1)\,\frac{1}{R^2}\,\|\vartheta''\|_\infty \,\|f_n\|_{X_m}.
	\end{align*}
	Letting $\ve\to 0$, Lebesgue's dominated convergence theorem entails
	\begin{align*}
		\int_0^\infty x^{m}\,&\vartheta_R(x)\,\vert f_n(x)\vert \, \big(\omega(x)\wedge n\big)\, \mathrm{d}x\\
		&\le \int_R^\infty x^m\, \vert g_n(x)\vert \, \mathrm{d}x
		+ (2m-\kappa+1)\,\frac{1}{R^2}\,\|\vartheta''\|_\infty \,\|f_n\|_{X_m},
	\end{align*}
	and therefore
	\begin{align*}
		\int_{2R}^\infty x^{m}\,  \big(\omega(x)\wedge n\big)\, \vert f_n(x)\vert \,\mathrm{d}x&\le \int_{0}^\infty x^{m}\,\vartheta_R(x)\, \vert f_n(x)\vert \, \big(\omega(x)\wedge n\big)\, \mathrm{d}x\\
		&\le \int_R^\infty x^m\, \vert g_n(x)\vert \, \mathrm{d}x
		+ (2m-\kappa+1)\,\frac{1}{R^2}\,\|\vartheta''\|_\infty \,\|f_n\|_{X_m}.
	\end{align*}
	Invoking~\eqref{eq1} we thus infer that
	\begin{align}\label{eq2}
		\lim_{R\to\infty}\, \sup_{n\in\mathbb{N}} \left( \int_{2R}^\infty x^{m}\,  \big(\omega(x)\wedge n\big)\,\vert f_n(x)\vert \, \mathrm{d}x\right)=0.
	\end{align}
	Now, given $R>1$ we find $N_R\ge 1$ such that $(\omega\wedge n)=\omega$ on $(0,2R)$ for $n\ge N_R$  thanks to~\eqref{a1}. Hence, for $n\ge N_R$,
	\begin{align*}
		\|(\omega\wedge n) f_n-\omega f\|_{X_m}&\le \int_0^{2R} x^{m}\, \omega(x)\,  \vert f_n(x)- f(x)\vert \, \mathrm{d}x +\int_{2R}^\infty x^{m}\, \big(\omega(x)\wedge n\big)\, \vert f_n(x)\vert \, \mathrm{d}x\\
		&\quad +\int_{2R}^\infty x^{m}\, \omega(x)\,  \vert f(x)\vert \, \mathrm{d}x .
	\end{align*}
	We then pass to the limit as $n\to\infty$ and infer from \eqref{eq1} that
	\begin{align*}
		\limsup_{n\to\infty} \|(\omega\wedge n) f_n - \omega f\|_{X_{m}} &\le \sup_{l\ge 1} \left\{ \int_{2R}^\infty x^m\ \big(\omega(x)\wedge l\big)\, |f_l(x)|\ \mathrm{d}x \right\}\\
		&\qquad + \int_{2R}^\infty x^m\, \omega (x) |f(x)|\ \mathrm{d}x.
	\end{align*}
	We finally let $R\to\infty$ and use \eqref{eq2} along with $\omega f\in X_m$ to obtain
	\begin{equation}
		\lim_{n\to\infty} \|(\omega\wedge n) f_n - \omega f\|_{X_{m}} = 0. \label{A.7}
	\end{equation}
	
	\medskip
	
	\noindent\textbf{(iii)} To conclude the proof, observe that $f_n\in \dom(A_{1,m})$ with $f_n\to f$ in $X_m$ and
    $$
	A_{1,m}f_n=-g_n+(\omega\wedge n) f_n \mathop{\longrightarrow}_{n\to\infty} -g+\omega f \quad \text{ in }\ X_m
	$$
	thanks to~\eqref{eq1} and~\eqref{A.7}. Since $A_{1,m}$ is a closed operator in $X_m$ due to \Cref{P1}, we conclude that $f\in \dom(A_{1,m})$ with $A_{1,m}^\omega f=-g$ as claimed.
\end{proof}

\begin{thm}\label{T4}
	Assume that $\omega$ satisfies~\eqref{a1} and let $\kappa<m\le 1$. Then  $A_{1,m}^\omega$, defined in \eqref{Al}, generates a positive analytic contraction semigroup  on $X_m$. 
\end{thm}

\begin{proof}
	Note that $X_{m}$ is a Banach lattice with order-continuous norm  (see \cite[Chapter~4]{AlBu2006}) and recall that the disjoint complement $F^\perp$ of a subset $F$ of  the vector lattice $X_{m}$ is defined as
	\begin{equation*}
		F^\perp := \big\{ g \in X_{m}\ \big |\ \min\{|f|,|g|\} = 0 \;\text{ for all }\; f\in F\big\}.
	\end{equation*}
	Since $\mathcal{D}\bigl((0,\infty)\bigr)$ is a subset of $\dom(A_{1,m}^\omega)$, we immediately see that $ \left( \dom(A_{1,m}^\omega) \right)^\perp=\{0\}$. Consequently, we are in a position to apply \cite[Proposition~4.3]{ArBa1993} and conclude that there is an extension $\widehat{A_{1,m}^\omega}$  of $A_{1,m}^\omega$ with domain
	\begin{multline*}
		D(\widehat{A_{1,m}^\omega}):= 
		\Big\{ f\in X_{m}\,\Big |\, \text{ there are }(f_n)_{n\in\mathbb{N}}\text{ in }\dom(A_{1,m})\text{ and }g\in X_{m} \text{ such that } \\
		\lim_{n\to\infty} \left(\|f_n-f\|_{X_{m}} + \|A_{1,m}f_n - (\omega\wedge n) f_n + g\|_{X_{m}} \right) = 0 \Big\} 
	\end{multline*}
	which generates a positive strongly continuous semigroup on $X_m$. \Cref{LL1} implies that
	$D(\widehat{A_{1,m}^\omega})=\dom(A_{1,m}^\omega)$ and therefore that $\widehat{A_{1,m}^\omega}=A_{1,m}^\omega$. Moreover, it holds that
	$$
        0\le e^{t A_{1,m}^\omega}=e^{t \widehat{A_{1,m}^\omega}}\le e^{t A_{1,m}},\quad t\geq 0,
    $$
 by \cite[p.~432]{ArBa1993}. Since $A_{1,m}$ generates a positive, analytic contraction semigroup, this ordering property, along with \cite[Remark~2.68]{BaAr2006}, implies that
	$$
	\| e^{t {A_{1,m}^\omega}}\|_{\mathcal{L}(X_m)}\le \| e^{t A_{1,m}}\|_{\mathcal{L}(X_{m})}\le 1,\quad t\ge 0.
	$$ 
	Hence $A_{1,m}^\omega=\widehat{A_{1,m}^\omega}$ generates a positive contraction semigroup on $X_m$ which is also analytic due to \cite[Theorem~6.1]{ArBa1993}.
\end{proof}

\begin{cor}\label{C4}
	Assume that $\omega$ satisfies~\eqref{a1} and let $(\alpha,a,b,n)\in\mathbb{R}^4$ satisfy, either the assumptions of \Cref{T5}~(c2), or those of \Cref{T6}~(c2). Then  $\mathbb{A}_n^\omega$, defined by
	\begin{equation*}
			\mathbb{A}_n^\omega f:=\mathbb{A}_n f-\omega f,\quad f\in \dom(\mathbb{A}_n^\omega),
	\end{equation*}
	with domain
	\begin{equation*}  
			\dom(\mathbb{A}_n^\omega):=\big\{f\in \dom(\mathbb{A}_n)\, \big |\,  \omega f\in X_n\big\},
	\end{equation*}
	generates a positive analytic contraction semigroup  on $X_n$. 
\end{cor}

\begin{proof}
	As in \Cref{L16}, we deduce from \Cref{L14} that
	\begin{equation*}
		\mathcal{A}_\alpha(a,b) - \omega = \frac{(2-\alpha)^2}{4}\, \big(\Tk_{-\alpha/2}\mathcal{M}_{l_\pm}\big)\, \big( \mathcal{G}_{\kappa_\pm} - \widetilde{\omega}\big)\,\big(\Tk_{-\alpha/2}\mathcal{M}_{l_\pm}\big)^{-1}
	\end{equation*}
	with
    \begin{equation*}
		\widetilde{\omega} (x) := \frac{4}{(2-\alpha)^2}\,\omega\left(x^{2/(2-\alpha)}\right)\,,\quad x>0\,.
	\end{equation*}
The statement now follows from \Cref{T4} as in the proof of \Cref{{T5}}.
\end{proof}

\section*{Acknowledgements}
For the purpose of Open Access, a CC-BY public copyright licence has been applied by the authors to the present document and will be applied to all subsequent versions up to the Author Accepted Manuscript arising from this submission.

\bibliographystyle{siam}
\bibliography{SingOps}

%
\appendix
%

\section{Estimates for Certain Integral Operators}

Let $s>0$ be fixed. Given $(\alpha, \beta)\in \R^2$, we define the integral operator 
$$
\big(Q_{s,\alpha,\beta}(t) f\big)(x):=\int_0^\infty q_{s,\alpha,\beta}(t,x,r) \, f(r)\,\rd r,\quad (t,x)\in (0,\infty)^2,
$$
whenever it is well-defined, where
\begin{equation}\label{k11x}
q_{s,\alpha,\beta}(t,x,r):=\frac{1}{\sqrt{t}} \left(\frac{x}{\sqrt{t}}\wedge 1\right)^{-\alpha}\left (\frac{r}{\sqrt{t}}\wedge 1\right)^{-\beta}\,\exp\left(-\frac{\vert x-r\vert^2}{s t}\right)
\end{equation}
for $(t,x,r)\in (0,\infty)^3$.

\begin{prop}\label{P11x}
Let $(\alpha, \beta)\in \R^2$, $m\in (-\infty,1]$, and $\theta\ge 0$ satisfy $\theta+\alpha<m+1\le 1-\beta$.  Then, there is $c>0$ depending on $s$, $\alpha$, $\beta$, $m$, and $\theta$ such that
\begin{equation*} 
\| Q_{s,\alpha,\beta}(t)\|_{\ml(X_m,X_{m-\theta})}\le c  t^{-\theta/2},\quad t>0.
\end{equation*}
\end{prop}

\begin{proof}
We proceed along the lines of \cite[Appendix~C]{MNS_JDE22}. Let $f\in X_m$ and $t>0$. Then
\begin{align*}
\| Q_{s,\alpha,\beta}(t) f\|_{X_{m-\theta}}& \le \int_0^\infty \int_0^\infty   q_{s,\alpha,\beta}(t,x,r) \, \vert f(r)\vert \,\rd r\, x^{m-\theta}\,\rd x\\
&\le t^{\frac{m-\theta+1}{2}} \int_0^\infty\int_0^\infty  \left(x\wedge 1\right)^{-\alpha}\left(r\wedge 1\right)^{-\beta}\,\exp\left(-\frac{\vert x-r\vert^2}{s}\right)\, \left\vert f\left(\sqrt{t} r\right)\right\vert x^{m-\theta}\, \rd r\, \rd x\\
&\le t^{\frac{m-\theta+1}{2}} \int_0^1 x^{-\alpha+m-\theta}\,\rd x \int_0^1 r^{-\beta}\,   \left\vert f\left(\sqrt{t} r\right)\right\vert\, \rd r\\
&\quad +t^{\frac{m-\theta+1}{2}} \int_1^\infty\int_0^1  r^{-\beta} \,\exp\left(-\frac{\vert x-r\vert^2}{s}\right)\, \left\vert f\left(\sqrt{t} r\right)\right\vert x^{m-\theta}\, \rd r\, \rd x\\
&\quad+ t^{\frac{m-\theta+1}{2}} \int_0^1\int_1^\infty  x^{-\alpha+m-\theta}\,\exp\left(-\frac{\vert x-r\vert^2}{s}\right)\, \left\vert f\left(\sqrt{t} r\right)\right\vert \, \rd r\, \rd x\\
&\quad+ t^{\frac{m-\theta+1}{2}} \int_1^\infty\int_1^\infty  \exp\left(-\frac{\vert x-r\vert^2}{s}\right)\, \left\vert f\left(\sqrt{t} r\right)\right\vert x^{m-\theta}\, \rd r\, \rd x\\
&=:I+II+III+IV.
\end{align*}
In the following, the constants $c$ may depend on $s$, $\alpha$, $\beta$, $m$, and $\theta$, but are independent of $t>0$ and $f\in X_m$. Since $\theta+\alpha<m+1\le 1-\beta$ we have
\begin{align*}
I&\le t^{\frac{m-\theta+1}{2}}\, \int_0^1 x^{-\alpha+m-\theta}\,\rd x\, \int_0^1 r^m\,   \left\vert f\left(\sqrt{t} r\right)\right\vert\, \rd r \\
&\le  c\,t^{\frac{m-\theta+1}{2}}\, \| f(\sqrt{t}\,\cdot)\|_{X_m}= c\,t^{-\theta/2}\, \| f \|_{X_m}.
\end{align*}
Next, since $r^{-\beta}\le r^m$ and $(x-r)^2\ge x$ for $x\ge 3$ and $0\le r\le 1$, we have
\begin{align*}
II&\le t^{\frac{m-\theta+1}{2}}   \int_1^3 x^{m-\theta}\, \rd x \int_0^1  r^m\,  \left\vert f\left(\sqrt{t} r\right)\right\vert \, \rd r\\
&\qquad + t^{\frac{m-\theta+1}{2}} 
 \int_3^\infty x^{m-\theta}\, \exp\left(-\frac{x}{s}\right)\, \rd x \int_0^1  r^m\, \left\vert f\left(\sqrt{t} r\right)\right\vert\, \rd r  \\
&\le c \,t^{\frac{m-\theta+1}{2}}\, \| f(\sqrt{t}\,\cdot)\|_{X_m}=c\,t^{-\theta/2}\, \| f \|_{X_m}.
\end{align*}
Similarly, since $\theta+\alpha<m+1$, we derive
\begin{align*}
III&\le t^{\frac{m-\theta+1}{2}} \max\{1,3^{-m}\} \int_0^1 x^{-\alpha+m-\theta}\, \rd x \int_1^3 r^m\, \left\vert f\left(\sqrt{t} r\right)\right\vert \, \rd r\\
&\qquad + t^{\frac{m-\theta+1}{2}}\sup_{r\ge 3}\Big[r^{-m}\, \exp{\left( -\frac{r}{s} \right)}\Big]\int_0^1 x^{-\alpha+m-\theta}\, \rd x \int_3^\infty r^m\, \left\vert f\left(\sqrt{t} r\right)\right\vert \, \rd r\\
&\le c \,t^{\frac{m-\theta+1}{2}}\, \| f(\sqrt{t}\,\cdot)\|_{X_m}=c\,t^{-\theta/2}\, \| f \|_{X_m}.
\end{align*}
We finally deal with $IV$, where the non-negativity of $\theta$ ensures that
\begin{align*}
IV&\le t^{\frac{m-\theta+1}{2}} \int_1^\infty \left\vert f\left(\sqrt{t} r\right)\right\vert\, \int_{1-r}^\infty  \exp\left(-\frac{x^2}{s}\right)\, (x+r)^m\, \rd x\, \rd r.
\end{align*}
If $m\in [0,1]$, then 
\begin{align*}
\int_{1-r}^\infty  \exp\left(-\frac{x^2}{s}\right)\, (x+r)^m\, \rd x\le  \int_{-\infty}^\infty  \exp\left(-\frac{x^2}{s}\right)\, (\vert x\vert^m+r^m)\, \rd x\le c\, r^m
\end{align*}
for $r\ge 1$ so that
\begin{align*}
IV&\le c \,t^{\frac{m-\theta+1}{2}}\, \| f(\sqrt{t}\,\cdot)\|_{X_m}=c\,t^{-\theta/2}\, \| f \|_{X_m}.
\end{align*}
If $m< 0$, then, for $r\ge 1$ and $x\ge 1-r$,
$$
\left( \frac{r}{r+x} \right)^{-m} \le \left( \frac{r+x+|x|}{r+x} \right)^{-m} \le \big( 1 + |x| \big)^{-m} \le 2^{-m} \big( 1 + |x|^{-m} \big),
$$
from which we deduce that
\begin{align*}
\int_{1-r}^\infty  \exp\left(-\frac{x^2}{s}\right)\, (x+r)^m\, r^{-m}\, \rd x\le 
2^{-m} \int_{-\infty}^\infty  \exp\left(-\frac{x^2}{s}\right)\, \big(1+\vert x\vert^{-m}\big)\, \rd x \le c
\end{align*}
for $r\ge 1$.  We thus again find that
\begin{align*}
IV&\le c \,t^{\frac{m-\theta+1}{2}}\, \| f(\sqrt{t}\,\cdot)\|_{X_m}=c\,t^{-\theta/2}\, \| f \|_{X_m}.
\end{align*}
In summary, we have shown that
\begin{equation*}
\| Q_{s,\alpha,\beta}(t) f\|_{X_{m-\theta}}\le c\,t^{-\theta/2}\, \| f \|_{X_m}
\end{equation*} 
for $f\in X_m$ and  $t>0$, which establishes the claim.
\end{proof}

\section{Auxiliary Results}\label{S2}

The aim of this section is to describe the behavior of functions $f\in X_m$ with $\Gk f\in X_m$ for some $\kappa\in (-\infty,1\wedge m)$ as $x\to 0$ and as $x\to\infty$ .

\begin{lem}\label{L1}
	Given  $\kappa\in (-\infty,1)$ and $m> \kappa$, assume that $f\in X_m$ with $\Gk f\in X_m$. 
	Then
	\begin{equation}\label{f1x}
		\lim_{x\to \infty} x^{m-1}f(x)=0.
	\end{equation}
	Moreover,  $f'\in X_{m-1}$ with
	\begin{equation}\label{f2ax}
		f'(x)=-x^{-\kappa}\int_x^\infty z^\kappa\, \Gk f(z)\,\rd z,\quad x>0,
	\end{equation}
	and
	\begin{equation}\label{f2x}
		\lim_{x\to \infty} x^{m}f'(x) = \lim_{x\to 0} x^{m}f'(x) = 0.
	\end{equation}
	Finally, for $m\ge 1$ it holds that
	\begin{equation}\label{f3x}
		f(x)=\int_x^\infty z^{\kappa}\,\Gk f(z)\frac{z^{1-\kappa}-x^{1-\kappa}}{1-\kappa}\,\rd z,\quad x>0.
	\end{equation}
\end{lem}

\begin{proof}
	\noindent {\bf (i)} First consider the case when $m\ge 1$, so that $\kappa <1\le m$, and let
	\begin{equation}\label{f3xx}
		F(x):=\int_x^\infty z^{\kappa}\,\Gk f(z)\frac{z^{1-\kappa}-x^{1-\kappa}}{1-\kappa}\,\rd z,\quad x>0,
	\end{equation}
	which is well-defined since $\Gk f\in X_m$ and since
	$$
	\left\vert z^{\kappa}\,\Gk f(z)\frac{z^{1-\kappa}-x^{1-\kappa}}{1-\kappa}\right\vert \le \frac{z}{1-\kappa}\,\vert\Gk f(z)\vert\le \frac{x^{1-m}}{1-\kappa}\,z^m\,\vert\Gk f(z)\vert,
	$$
 for $z>x>0$. Moreover, one has that 
	$$
	x^{m-1}\,F(x)\le \frac{1}{1-\kappa}\int_x^\infty z^m\, \vert\Gk f(z)\vert\, \rd z ,\quad x>0,
	$$
	from which one infers that 
	\begin{equation}\label{f1xa}
		\lim_{x\to \infty} x^{m-1}F(x)=0.
	\end{equation}
	Next observe that $F\in W_{1,loc}^1((0,\infty))$ and that
	\begin{equation}\label{f'id}
		F'(x)=-x^{-\kappa}\int_x^\infty z^\kappa\, \Gk f(z) \, \rd z,\quad x>0.
	\end{equation}
	In particular, $[x\mapsto x^\kappa F'(x)]\in W_{1,loc}^1((0,\infty))$ satisfies
	$$
	\big(x^\kappa F'(x)\big)'=x^\kappa  \Gk f(x)=\big(x^\kappa f'(x)\big)',\quad x>0.
	$$
	Integrating this identity yields
	$$
	f(x)=B+A\frac{x^{1-\kappa}}{1-\kappa} + F(x),\quad x>0,
	$$
	for some $(A,B)\in \R^2$. Since 
	$$
	\int_x^{x+1} y^{m-1} \vert f(y)\vert \,\rd y  \le\frac{1}{x} \int_x^{x+1} y^{m} \vert f(y)\vert\,\rd y
	$$
	and
	\begin{align*}
	\int_x^{x+1} y^{m-1} \vert F(y)\vert \,\rd y  & \le  \frac{1}{1-\kappa}\int_x^{x+1}\int_y^\infty z^m\, \vert \Gk f(z)\vert \,\rd z\,\rd y \\
    & \le \frac{1}{1-\kappa} \int_x^{\infty} z^{m} \vert \Gk f(z)\vert\,\rd z,
	\end{align*}
	we deduce from $f\in X_m$ and $\Gk f\in X_m$ that
	$$
	\lim_{x\to\infty}  \int_x^{x+1} y^{m-1} \vert f(y)-F(y)\vert\,\rd y=0
	$$
	and, hence, that
	$$
	\lim_{x\to\infty}  \int_x^{x+1} y^{m-1}\left( B+A\frac{y^{1-\kappa}}{1-\kappa} \right) \,\rd y=0.
	$$
	However,
	\begin{align*}
		\int_x^{x+1} y^{m-1}\left(B+A\frac{y^{1-\kappa}}{1-\kappa} \right) \,\rd y&=\frac{B}{m}\left[(x+1)^m-x^m\right]+\frac{A \left[(x+1)^{m+1-\kappa}-x^{m+1-\kappa}\right]}{(1-\kappa)(m+1-\kappa)}\\
		&\sim \frac{A}{1-\kappa} x^{m-\kappa}\ \text{ as } x\to \infty,
	\end{align*}
	which implies $A=0$ since $m>\kappa$. Then
	\begin{align*}
		\int_x^{x+1} y^{m-1}\left(B+A\frac{y^{1-\kappa}}{1-\kappa} \right) \,\rd y&=\frac{B}{m}\left[(x+1)^m-x^m\right] \sim B x^{m-1} \text{ as }x\to \infty,
	\end{align*}
	and, hence, also $B=0$, since $m\ge 1$. Consequently, $f=F$ and, thanks to~\eqref{f3xx}, \eqref{f1xa} and~\eqref{f'id}, we see that \eqref{f3x} holds, along with \eqref{f1x} and~\eqref{f2ax}, for $m\ge 1$.\\[.15cm]
	\noindent {\bf (ii)} Now consider the case when $m<1$, so that $\kappa <m<1$. Define
	\begin{equation}\label{f17}
		F(x):=\int_1^x z^{\kappa}\,\Gk f(z)\,\frac{x^{1-\kappa}-z^{1-\kappa}}{1-\kappa}\,\rd z,\quad x>0,
	\end{equation}
	which is meaningful since
	$$
	\left\vert z^\kappa \Gk f(z)\frac{x^{1-\kappa}-z^{1-\kappa}}{1-\kappa}\right\vert\le  \frac{x^{1-\kappa}}{1-\kappa}\, \vert z^m \Gk f(z)\vert,\quad x> z> 1,
	$$
	and
	$$
	\left\vert z^\kappa \Gk f(z)\frac{x^{1-\kappa}-z^{1-\kappa}}{1-\kappa}\right\vert\le  \frac{1}{1-\kappa}\, \vert z^m \Gk f(z)\vert,\quad x< z< 1.
	$$
	As in the previous case, $F\in W_{1,loc}^1((0,\infty))$ satisfies
	$$
	F'(x) = x^{-\kappa} \int_1^x z^\kappa \mathcal{G}_\kappa f(z)\, \rd z, \quad x>0,
	$$
	and $[x\mapsto x^\kappa F'(x)]\in W_{1,loc}^1((0,\infty))$ with
	$$
	\big(x^\kappa\, F'(x)\big)'=x^\kappa \,\Gk f(x)=\big(x^\kappa f'(x)\big)',\quad x>0,
	$$
	so that there are $(A,B)\in \R^2$ such that
	\begin{equation}\label{iu1}
		f(x)=B+A\frac{x^{1-\kappa}}{1-\kappa} + F(x),\quad x>0.
	\end{equation}
	Then
	\begin{equation}\label{iu}
		f'(x)=x^{-\kappa}\left(A+\int_1^x z^\kappa\Gk f(z)\,\rd z\right),\quad x>0,
	\end{equation}
	and consequently
	$$
	\lim_{x\to\infty} x^\kappa f'(x)= A+\int_1^\infty z^\kappa\,\Gk f(z)\,\rd z=:L <\infty
	$$
	since
	$[z\mapsto z^\kappa\Gk f(z)]\in L_1\bigl((1,\infty)\bigr)$ thanks to $m>\kappa$. Towards a contradiction, assume that  $L\not=0$. An integration shows that 
	$$
	x^m f(x)\sim \frac{L}{1-\kappa}x^{m+1-\kappa} \text{ as }x\to \infty,
	$$
	contradicting the integrability of $f\in X_m$. Thus, $L=0$ and
	\begin{equation}\label{i18}
		A=-\int_1^\infty z^\kappa\Gk f(z)\,\rd z,
	\end{equation}
	so that~\eqref{iu} implies~\eqref{f2ax}. Let us now establish~\eqref{f1x}. Using~\eqref{f17}, \eqref{iu1}, and~\eqref{i18}, it follows that
	\begin{equation}\label{iu3}
		f(x)=B-\frac{x^{1-\kappa}}{1-\kappa}\int_x^\infty z^\kappa\,\Gk f(z)\,\rd z - \frac{1}{1-\kappa}\int_1^x z\,\Gk f(z)\,\rd z,\quad x>0.
	\end{equation}
	Notice that
	\begin{align*}
		x^{m-1}\int_1^x z\,\vert\Gk f(z)\vert\, \rd z&= x^{m-1} \int_1^R z^{1-m}\, z^m\, \vert\Gk f(z)\vert\,\rd z + x^{m-1} \int_R^x z^{1-m}\, z^m\, \vert\Gk f(z)\vert\,\rd z  \\
		&\le x^{m-1}\, R^{1-m}\int_1^R  z^{m}\,  \vert\Gk f(z)\vert\,\rd z + \int_R^x  z^m\, \vert\Gk f(z)\vert\,\rd z,
	\end{align*}
  for $1<R<x$. Since $m<1$ and $\Gk f\in X_m$, we see that
	\begin{align*}
		\limsup_{x\to \infty} x^{m-1}\int_1^x z\,\vert\Gk f(z)\vert\, \rd z\le \int_R^\infty  z^m\, \vert\Gk f(z)\vert\,\rd z.
	\end{align*}
	Hence, letting $R\to\infty$, it follows that
	\begin{align}\label{u13}
		\lim_{x\to \infty} x^{m-1}\int_1^x z\,\vert\Gk f(z)\vert\, \rd z=0.
	\end{align}
	Since $\kappa<m<1$ and
	$$
	\left\vert x^{m-1} x^{1-\kappa}\int_x^\infty z^\kappa\,\Gk f(z)\,\rd z\right\vert \le  \int_x^\infty z^m\,\vert \Gk f(z)\vert\,\rd z \longrightarrow 0\text{ as }x\to\infty,
	$$
	we deduce~\eqref{f1x} from \eqref{iu3} and \eqref{u13}. We have thus established~\eqref{f1x} and~\eqref{f2ax} for $m<1$ as well.\\[.15cm]
	\textbf{(iii)} It remains to prove~\eqref{f2x} and that $f'\in X_{m-1}$. The latter is ensured by
	\begin{align*}
		\int_0^\infty x^{m-1}\, \vert f'(x)\vert\,\rd x\le \int_0^\infty x^{m-1-\kappa}\int_x^\infty z^\kappa\, \vert \Gk f(z)\vert\,\rd z\, \rd x=\frac{1}{m-\kappa}\int_0^\infty z^m\,\vert\Gk f(z)\vert\,\rd z ,
	\end{align*}
	which follows from~\eqref{f2ax}. Using once more~\eqref{f2ax} and $m>\kappa$, we find
	\begin{align*}
		x^m |f'(x)| \le x^{m-\kappa} \int_x^\infty z^\kappa\, \vert \Gk f(z)\vert\,\rd z \le \int_x^\infty z^m\, \vert \Gk f(z)\vert\,\rd z,
	\end{align*}
	from which we deduce that
	\begin{equation*}
		\lim_{x\to\infty} x^m f'(x) = 0
	\end{equation*}
	since $\Gk f\in X_m$. Finally, let $\delta\in (0,1)$ and consider $x\in (0,\delta)$. Then, again using~\eqref{f2ax} and the fact that $m>\kappa$, we see that
	\begin{align*}
		x^m\,\vert f'(x)\vert&\le x^{m-\kappa} \int_x^\delta z^{\kappa-m}\, z^m\, \vert\Gk f(z)\vert\,\rd z + x^{m-\kappa} \int_\delta^\infty z^{\kappa-m}\, z^m\, \vert\Gk f(z)\vert\,\rd z \\
		&\le \int_x^\delta  z^{m}\,  \vert\Gk f(z)\vert\,\rd z +x^{m-\kappa}\,\delta^{\kappa-m} \int_\delta^\infty  z^m\, \vert\Gk f(z)\vert\,\rd z ,
	\end{align*}
	so that  $\Gk f\in X_m$ implies that
	\begin{align*}
		\limsup_{x\to 0}x^m\,\vert f'(x)\vert&\le  \int_0^\delta  z^{m}\,  \vert\Gk f(z)\vert\,\rd z.
	\end{align*}
	Letting $\delta\to 0$ completes the proof of~\eqref{f2x} and thus of \Cref{L1}.
\end{proof}

In order to account for the  different behavior of the cases when $m<1$ and when $m=1$ we define
$$
\widehat{X}_m:=\left\{\begin{array}{ll} \Y_{[m-2,m]},& m<1,\\
	X_1, &m=1.\end{array}
\right.
$$

\begin{cor}\label{C1}
	Let $\kappa < m\le 1$. If $f\in \widehat{X}_m$ and $\Gk f\in X_m$, then 
	\begin{equation}\label{f5}
		x^{m-1}f(x)=-\int_x^\infty\big(z^{m-1}f(z)\big)'\,\rd z,\quad x>0,
	\end{equation}
	and $\lim_{x\to 0} x^{m-1}f(x)$ exists.
\end{cor}

\begin{proof}
	If $f\in \widehat{X}_m$ satisfies $\Gk f\in X_m$, then $f'\in X_{m-1}$ by \Cref{L1} and 
	$$
	\big(x^{m-1} f(x)\big)' = x^{m-1} f'(x) + (m-1) x^{m-2} f(x),\quad x>0,
	$$ 
	so that 
	$$
	[x\mapsto (x^{m-1}f(x))']\in L_1\bigl((0,\infty)\bigr).
	$$
	Hence, using~\eqref{f1x}, we deduce that
	\begin{equation*}
		x^{m-1}f(x)=-\int_x^\infty\big(z^{m-1}f(z)\big)'\,\rd z,\quad x>0,
	\end{equation*}
	and that $\lim_{x\to 0} x^{m-1}f(x)$ exists. 
\end{proof}

\end{document}